\documentclass {article}
\usepackage[english]{babel}
\usepackage[utf8]{inputenc}
\usepackage[T1]{fontenc}
\usepackage{graphicx,color}
\usepackage[final]{pdfpages}
\usepackage[extra]{tipa}
\usepackage{titlesec}
\usepackage{xspace}		
\usepackage[svgnames]{xcolor}
 \usepackage{lmodern}
\usepackage[hypertexnames=false]{hyperref}

\usepackage{mathtools}
\usepackage{graphicx}
\usepackage{amsmath,amsfonts}
\usepackage{amsthm}
\usepackage{dsfont}
\usepackage{amssymb}
\usepackage{mathenv}
\usepackage{nicefrac}

\usepackage{caption,subcaption}
\usepackage{color}
\usepackage{cite}
\mathtoolsset{showonlyrefs,showmanualtags}

\usepackage{aliascnt}

\newtheorem{theorem}{Theorem}
\newtheorem{proposition}{Proposition}
\newtheorem{corollary}{Corollary}
\newtheorem{lemma}{Lemma}
\newtheorem{definition}{Definition}
\newtheorem{remark}{Remark}

\newcommand{\D}{\partial}
\newcommand{\eps}{\varepsilon}
\newcommand{\R}{\mathbb{R}}
\newcommand{\la}{\lambda}

\newtheorem{manualtheoreminner}{Theorem}
\newenvironment{manualtheorem}[1]{%
  \renewcommand\themanualtheoreminner{#1}%
  \manualtheoreminner
}{\endmanualtheoreminner}

\begin{document}

\title{Non-existence of patterns and gradient estimates}

\author{Samuel Nordmann\thanks{samnordmann@gmail.com ; School of Mathematical Sciences, Tel Aviv University.
}}



\maketitle
\begin{abstract}
We call pattern any non-constant solution of a semilinear elliptic equation with Neumann boundary conditions.
A classical theorem of Casten, Holland \cite{Casten1978a} and Matano \cite{Matano1979a} states that stable patterns do not exist in convex domains. 
In this article, we show that the assumptions of \emph{convexity of the domain} and \emph{stability of the pattern} in this theorem can be relaxed in several directions.
In particular, we propose a general criterion for the non-existence of patterns, dealing with possibly non-convex domains and unstable patterns.
Our results unfold the interplay between the geometry of the domain, the magnitude of the nonlinearity, and the stability of patterns.

In addition, we establish several gradient estimates for the patterns of~\eqref{ANG_Intro_EquationSemilineaire}. We prove a general nonlinear Cacciopoli inequality (or an inverse Poincaré inequality), stating that the $L^2$-norm of the gradient of a solution is controlled by the $L^2$-norm of $f(u)$, with a constant that only depends on the domain. 
This inequality holds for non-homogeneous equations.
We also give several flatness estimates.

Our approach relies on the introduction of what we call the \emph{Robin-curvature Laplacian}. This operator is intrinsic to the domain and contains much information on how the geometry of the domain affects the shape of the solutions.

Finally, we extend our results to unbounded domains. It allows us to improve the results of our previous paper~\cite{Nordmann2019a} and to extend some results on De Giorgi's conjecture to a larger class of domains.

\end{abstract}

\noindent {\bf AMS Class. No:} 35B35, 35B06, 35J15, 35J61, 35B53.
\paragraph{Acknowledgement.}
The author is deeply thankful to Professor Henri Berestycki for many very instructive discussions.\\
The research leading to these results has received funding from the European
Research Council under the European Union's Seventh Framework Programme
(FP/2007-2013) / ERC Grant Agreement n.321186 - ReaDi -ReactionDiffusion
Equations, Propagation and Modelling held by Henri Berestycki.

\section{Introduction}

We call \emph{pattern} any non-constant solution of the following equation
\begin{equation}\label{ANG_Intro_EquationSemilineaire}
    \left\{
    \begin{aligned}
        &-\Delta u(x)=f(u(x)) &&\forall x\in\Omega,\\ 
        &\D_\nu u(x)=0 &&\forall x\in\D\Omega,\\
    \end{aligned}
    \right.
\end{equation}
where the domain $\Omega\subset\R^n$ ($n\geq2$) is bounded and smooth (say, $C^2$), $\D_\nu$ is the outward normal derivative, and the nonlinearity $f$ is $C^{1}$. 
We sometimes consider the non-homogeneous case, that is, we sometimes allow $f=f(x,u)$ to depend on $x\in\Omega$.

Note that if $z\in\R$ is a root of $f$, then $z$ is a solution of the above equation (but it is not a pattern).


\paragraph*{}
We are interested in the existence/non-existence and gradient estimates of patterns.
Our starting point is a celebrated theorem proved independently by Casten, Holland~\cite{Casten1978a}, and by Matano~\cite{Matano1979a}. 
\begin{manualtheorem}{CHM}[\cite{Casten1978a,Matano1979a}]\label{ANG_th:Intro_CHM}
If the domain is convex, there exists no stable pattern to~\eqref{ANG_Intro_EquationSemilineaire}.
\end{manualtheorem}
Here, a solution is said to be \emph{stable} if the second variation of energy is nonnegative~(see \autoref{def:stability} below).

\autoref{ANG_th:Intro_CHM} states that, if the domain is convex, any stable solution is constant.
This result has many consequences on the classification of solutions, the asymptotics of the associated evolution problem, the construction of traveling fronts~\cite{Berestycki1992}, and has also many applications in chemistry, population dynamics, etc.
\autoref{ANG_th:Intro_CHM} has been extended in several directions: certain non-linear elliptic operators~\cite{Dipierro2018,Ciraolo2020}, Robin boundary conditions~\cite{Bandle2016,Dipierro2018}, manifolds~\cite{Jimbo1984,Bandle2012,Dipierro2018a}, unbounded convex domains~\cite{Nordmann2019a} and some systems~\cite{Jimbo1994,Yanagida2002}. See also~\cite{Dipierro2016}.
We give in Section~\ref{sec:Literature} a more detailed presentation of the literature on this topic.



\paragraph*{}
The present article aims to investigate whether the assumptions of \emph{convexity of the domain} and \emph{stability of the solution} can be relaxed in~\autoref{ANG_th:Intro_CHM}. To our knowledge, the literature does not contain any result in this direction. In particular, no sufficient condition is known for the non-existence of patterns in \emph{non-convex domains}.

Apparently, the assumptions of \autoref{ANG_th:Intro_CHM} are sharp in general.
On the one hand, the conclusion of \autoref{ANG_th:Intro_CHM} does not hold for unstable patterns, since $u(x)=\cos(x)$ is an unstable pattern of~\eqref{ANG_Intro_EquationSemilineaire} for $\Omega=x\in[0,\pi]$ and $f(u)=u$.
On the other hand, we cannot relax the assumption of convexity of the domain, since a result of Kohn and Sternberg~\cite{Kohn1989} allows constructing stable patterns in some star-shaped domains obtained as a small perturbation of a convex domain (such a counter-example can also be achieved in dimension $n\geq 3$ for a domain with positive mean curvature, see Section~\ref{sec:Literature} for more details).

Nevertheless, when the domain is almost convex, the above construction of a stable pattern needs the nonlinearity to have a large magnitude. This suggests that, in non-convex domains, the non-existence of patterns involves both the geometry of the domain and the magnitude of the nonlinearity.

\paragraph*{}
In this paper, we improve and extend~\autoref{ANG_th:Intro_CHM} in several directions. First, we show that the two assumptions of \emph{convexity of the domain} and \emph{stability of the patterns} can be combined in a single assumption on the spectrum of a certain operator (\autoref{ANG_Intro_th:GammaStability}). Then, we propose a general criterion for the non-existence of patterns (\autoref{th:CoroInstablePRELIMINARY}), which deals with possibly non-convex domains and possibly unstable patterns. 
To the best of our knowledge, this is the first result of this type.
This criterion unfolds the interplay between the geometry of the domain, the magnitude of the nonlinearity, and the stability of patterns.
We propose several applications, in particular, we prove that (under a geometric assumption) there exists no patterns if the domain is shrunk or if the nonlinearity has a small magnitude (\autoref{Coro_Shrinking}). As another consequence of our results, we prove that \autoref{ANG_th:Intro_CHM} is robust under smooth perturbations of the domain and the nonlinearity (\autoref{th:CoroPerturbation}). We also improve the conclusion of \autoref{ANG_th:Intro_CHM} by providing explicit sufficient conditions for the non-existence of unstable solutions in convex domains (\autoref{th:ApplicationConvex}).

In addition, we establish several gradient estimates for the patterns of~\eqref{ANG_Intro_EquationSemilineaire}. We prove a general nonlinear Cacciopoli inequality (or an inverse Poincaré inequality), stating that the $L^2$-norm of the gradient of a solution is controlled by the $L^2$-norm of $f(u)$, with a constant that only depends on the domain (\autoref{th:CoroStable}). 
This inequality holds for non-homogeneous equations (i.e., when $f=f(x,u)$ depends on $x\in\Omega$).
To the best of our knowledge, this result is completely new.
Also, we give several flatness estimates in term of the spectral gap of some geometric operator (\autoref{th:FlatnessEstimateRobinCurvatureCoro}) or the Morse index of the solution (\autoref{th:FlatnessEstimateRobinCurvature}).


Finally, we extend our main results to unbounded domains. This question involves additional difficulties and is closely linked to De Giorgi's conjecture. We recover and improve the results of our previous paper~\cite{Nordmann2019a} and extend some results on De Giorgi's conjecture to a larger class of domains.

\paragraph*{}
For simplicity, we only consider classical solutions of~\eqref{ANG_Intro_EquationSemilineaire}, however, our method relies on variational arguments and therefore most of our results hold for weak $H^1$ solutions.

We point out that the conclusion of \autoref{ANG_th:Intro_CHM} holds for any $C^1$ nonlinearity. Throughout the paper, we shall stick to this general context by making no further assumption on the nonlinearity.
However, let us mention that \autoref{ANG_th:Intro_CHM} (and subsequently our results) can be improved if the nonlinearity satisfies certain additional properties:
\begin{itemize}
\item  if $f$ is convex or concave, there exists no stable pattern in any (possibly non-convex) domain~\cite{Casten1978a}.
\item if $f'\leq0$ (resp. $f'>0$), any solution of~\eqref{ANG_Intro_EquationSemilineaire} is stable (resp. unstable), see \autoref{def:stability} below.
\item if $f\geq0$ or $f\leq 0$, there exists no (possibly unstable) patterns. This can be seen from integrating the equation~\eqref{ANG_Intro_EquationSemilineaire}, which implies $\int_\Omega f=0$, and so $f(u(x))=0$.
\item  if $\frac{f(u)}{u^{p^\star-1}}$ is non increasing, there exist no (possiby unstable) patterns in convex domains~\cite{Ciraolo2020}.
\end{itemize}

In our context, bistable nonlinearities (e.g. $f(u)=-(u-a)(u-b)(u-c)$, with $a<b<c$) are especially interesting since they favor the existence of stable patterns, in particular, the Allen-Cahn nonlinearity $f(u)=u-u^3$ which is balanced between its two stable roots (i.e. $\int_{-1}^{1} f=0$).

\paragraph*{}

In our approach, a key role is played by what we call the \emph{Robin-curvature Laplacian}.
\begin{definition}[Robin-curvature Laplacian]\label{Def:CurvatureGamma}
We denote by $\gamma(x)$ the lowest principal curvature of the domain at $x\in\D\Omega$ (i.e., $\gamma(x)$ is the lowest eigenvalue of the second fundamental form).  In dimension $n=2$, then $\gamma(\cdot)$ is nothing but the curvature of $\D\Omega$. 

We define the Robin-curvature Laplacian as the Laplace operator acting on $\psi\in C^2(\overline{\Omega})$ associated with the boundary conditions
\begin{equation}
\D_\nu\psi+\gamma \psi=0.
\end{equation}
We also define its lowest eigenvalue
\begin{equation}
\mu_{\gamma}:=\inf\limits_{\substack{\psi\in H^1\\ \Vert \psi\Vert_{{L}^2}=1}} \mathcal{G}_\gamma(\psi):=\inf\limits_{\substack{\psi\in H^1\\ \Vert \psi\Vert_{{L}^2}=1}}\int_\Omega \vert \nabla\psi\vert^2+\int_{\D\Omega}\gamma\psi^2,
\end{equation}
and more generally, for $a\geq0$,
\begin{equation}
\mu_{a\gamma}:=\inf\limits_{\substack{\psi\in H^1\\ \Vert \psi\Vert_{{L}^2}=1}} \mathcal{G}_{a\gamma}(\psi):=\inf\limits_{\substack{\psi\in H^1\\ \Vert \psi\Vert_{{L}^2}=1}}\int_\Omega \vert \nabla\psi\vert^2+a\int_{\D\Omega}\gamma\psi^2.
\end{equation}
\end{definition}
We emphasize that the Robin-curvature Laplacian is intrinsic to the domain and does not involve the nonlinearity. To the best of our knowledge, this operator has not been considered before in the literature. 
Note that $\gamma\geq0$ if and only if $\Omega$ is convex. Thus, if $\Omega$ is convex, the Robin-curvature Laplacian involves classical Robin boundary conditions with a nonnegative $\gamma$. If $\Omega$ is not convex, then $\gamma$ changes sign and the Robin boundary conditions are \emph{indefinite}. As a consequence of the work of Daners~\cite{Daners2013}, the Robin-curvature Laplacian satisfies most of the classical properties known for the \emph{definite Robin Laplacian}. In particular, its principal eigenvalue $\mu_\gamma$ is well defined and finite. Let us point out that the term $\int_{\D\Omega} \gamma\psi^2$ could be interpreted as the surface tension of a body $\Omega$ with rigidity on its surface.

Our results reveal that the Robin-curvature Laplacian contains much information on how the geometry of the domain affects the shape of the solutions of~\eqref{ANG_Intro_EquationSemilineaire}. An in-depth study of this operator could lead to new results.

\paragraph*{Outline.}
The results are presented in Section~\ref{sec:MainResults}. In Section~\ref{sec:Literature}, we discuss the existing literature and the implications of our results. The remaining of the paper is devoted to the proofs of our results. In Section~\ref{sec:Proofs_Non-existence}, we prove the results dealing with the non-existence of patterns (\autoref{ANG_Intro_th:GammaStability}, \autoref{th:CoroPerturbation} and \autoref{th:CoroInstablePRELIMINARY}). Section~\ref{sec:Cacciopoli} is devoted to the proof of the nonlinear Cacciopoli inequality (\autoref{th:CoroStable}). In Section~\ref{sec:flatnessEstimates}, we prove the flatness estimates (\autoref{th:FlatnessEstimateRobinCurvatureCoro} and \autoref{th:FlatnessEstimateRobinCurvature}). Finally, in Section~\ref{sec:UnboundedDomains} we extend our main results to unbounded domains (\autoref{th:Unbounded_Domains}, \autoref{th:CoroInstablePRELIMINARY_Unbounded} and \autoref{th:CoroStable_Unbounded}).

\section{Presentation of the results}\label{sec:MainResults}

\subsection{Non-existence of patterns}
The first step in our approach relies in the observation that the two assumptions of~\autoref{ANG_th:Intro_CHM} (that is, the stability of the solution and the convexity of the domain) can be combined in a single assumption on what we call the \emph{Robin-curvature-stability} of the solution.

Before presenting our first result, let us define precisely the \emph{stability} of a solution.
\begin{definition}[Stability]\label{def:stability}
A solution $u$ of~\eqref{ANG_Intro_EquationSemilineaire} is said to be {stable} if the second variation of energy at $u$ is nonnegative, namely, if
\begin{equation}\label{DefLambda}
\la_0:=\inf\limits_{\substack{\psi\in H^1\\ \Vert \psi\Vert_{{L}^2}=1}} \mathcal{F}_0(\psi):= \inf\limits_{\substack{\psi\in H^1\\ \Vert \psi\Vert_{{L}^2}=1}} \int_\Omega \vert \nabla \psi\vert^2-\int_\Omega f'(u)\psi^2\geq 0.
\end{equation}
\end{definition}
Note that $\lambda_0$ is the lowest eigenvalue of the linearized of operator of~\eqref{ANG_Intro_EquationSemilineaire} at $u$.
According to this definition, any local minimum or degenerate critical point of the energy is {stable}. Therefore, the set of stable solutions contains all the steady states obtained as a (strict) limit of an evolution problem.

Let us now define what we call the \emph{Robin-curvature-stability} of a solution.
\begin{definition}[Robin-curvature-stability]\label{Definition_Lambda_Gamma}
Let $u$ be a solution of~\eqref{ANG_Intro_EquationSemilineaire}. Set
\begin{equation}\label{ANG_Intro_DefLambdaGamma}
\la_\gamma:=\inf\limits_{\substack{\psi\in H^1(\Omega)\\                                                                                                                                                                                                                                                                                                                                                                                                          \Vert \psi\Vert_{{L}^2}=1}}\mathcal{F}_\gamma(\psi):=\inf\limits_{\substack{\psi\in H^1(\Omega)\\                                                                                                                                                                                                                                                                                                                                                                                                          \Vert \psi\Vert_{{L}^2}=1}} \int_\Omega \vert \nabla \psi\vert^2-\int_\Omega f'(u)\psi^2+\int_{\D\Omega}\gamma\psi^2,
\end{equation}
where $\gamma$ is the minimal curvature of the domain (\autoref{Def:CurvatureGamma}).
In analogy with \autoref{def:stability}, we say that $u$ is \emph{Robin-curvature-stable} if $\lambda_\gamma\geq0$.
\end{definition}
We point out that $\lambda_\gamma$ is the principal eigenvalue of the linearized operator at $u$, for which we have replaced the Neumann boundary conditions with the Robin-curvature boundary conditions $\D_\nu\psi+\gamma\psi=0$. Note that our notation is consistent with \autoref{def:stability} since $\lambda_\gamma$ coincides with $\lambda_0$ if we formally take $\gamma=0$ in the above definition. However, the Robin-curvature-stability do not refer to a usual notion of stability, because $u$ is not a critical point of the Robin-curvature energy $\mathcal{E}_\gamma(u)=\int_\Omega\frac{1}{2}\vert\nabla u \vert^2-F(u)+\int_{\D\Omega}\gamma u^2$ with $F'=f$.

We are ready to state a key observation.
\begin{theorem}\label{ANG_Intro_th:GammaStability}
There exists no Robin-curvature-stable patterns to~\eqref{ANG_Intro_EquationSemilineaire}.
\end{theorem}
We point out that this conclusion holds for possibly unstable patterns and non-convex domains.

\autoref{ANG_th:Intro_CHM} is a special case of the above theorem.
Indeed, recalling that $\Omega$ is convex if, and only if, $\gamma\geq0$, we see immediately that, for $u$ a solution of~\eqref{ANG_Intro_EquationSemilineaire},
\begin{equation}\label{MonotonieLambdaGammaPositif}
\text{if $\Omega$ is convex, then $\lambda_\gamma\geq \lambda_0$}
\end{equation}
(in fact, the strict inequality holds in~\eqref{MonotonieLambdaGammaPositif}, see \autoref{th:MonotonieStricteLambdaGammaPositif});
hence, in a convex domain, stability implies Robin-curvature-stability.

The Robin-curvature-stability ($\lambda_\gamma\geq0$) happens to combine the two assumptions of~\autoref{ANG_th:Intro_CHM}, namely, the stability ($\lambda_0\geq0$) and the convexity of the domain ($\gamma\geq0$). 
As a first consequence of \autoref{ANG_Intro_th:GammaStability} we establish the robustness of~\autoref{ANG_th:Intro_CHM} under smooth perturbations of the domain and the nonlinearity.
\begin{proposition}\label{th:CoroPerturbation}
Let $\Omega$ be a domain and $f$ a $C^1$ nonlinearity. If $(\Omega_\eps,f_\eps)$ is a smooth perturbation of $(\Omega,f)$, there exists no stable pattern to~\eqref{ANG_Intro_EquationSemilineaire}.
\end{proposition}
The perturbation is understood to be such that $(\Omega,f)\mapsto \lambda_\gamma$ is continuous. It is classical~\cite{Henry2005} (see also~\cite{Hale2005,Dancer1997,Daners1999}) that this is the case under a $C^1$ perturbation on the nonlinearity and a small $C^{2,\alpha}$ modification of the boundary of the domain (whereas it may fail under rougher perturbations~\cite{Daners}). 

As further explained in Section~\ref{sec:Literature}, a result of Kohn and Sternberg~\cite{Kohn1989} implies that there exist patterns in some domains that are perturbations of a convex domain (see \autoref{ANG_IntG_fig:SurfaceMinimale}). The above proposition implies that such a stable pattern does not exist unless the nonlinearity has a large $C^1$ norm.



\paragraph*{}
%
\autoref{ANG_Intro_th:GammaStability} involves the quantity $\lambda_\gamma$ which is rather implicit and depends on the solution itself. One may look for more explicit sufficient conditions for the non-existence of patterns.

The following result provides such a sufficient condition. It involves the stability of the solution, the geometry of the domain (through the Robin-curvature Laplacian) and the magnitude of the nonlinearity.
\begin{theorem}\label{th:CoroInstablePRELIMINARY}
Let $u$ be a pattern of~\eqref{ANG_Intro_EquationSemilineaire}. For any $a\geq 1$, we have
\begin{equation}\label{RelationStabilityGeometryMagnitude}
(a-1)\lambda_0< \sup f'-\mu_{a\gamma}.
\end{equation}
with $\mu_a$ from~\autoref{Def:CurvatureGamma}.
Consequently,
\begin{enumerate}
\item If $\mu_\gamma\geq \sup f'$, there exists no (possibly unstable) pattern to~\eqref{ANG_Intro_EquationSemilineaire}.
\item If $\mu_{a\gamma}\geq \sup f'$ for some $a\geq1$, there exists no stable pattern to~\eqref{ANG_Intro_EquationSemilineaire}.
\end{enumerate}
\end{theorem}
This theorem provides sufficent conditions for the non-existence of patterns when the domain is not convex or when the patterns are not supposed to be stable. Therefore, this theorem covers many cases that are not covered by \autoref{ANG_th:Intro_CHM}.

Note that \autoref{ANG_th:Intro_CHM} is included in the above theorem as a particular case (except for the case of degenerate stability $\lambda_0=0$): if $\Omega$ is convex (i.e. $\gamma\geq0$) and $u$ is a stable non-degenerate pattern (i.e. $\lambda_0>0$), then $\mu_{a\gamma}>0$ and, at the limit $a\to+\infty$ in~\eqref{RelationStabilityGeometryMagnitude}, we find a contradiction, hence $u$ is constant.
On the contrary, if $\Omega$ is not convex (i.e. $\gamma\not\geq0$), then $\lim_{a\to +\infty}\frac{\mu_{a\gamma}}{a}=-\infty$ (\cite{Levitin2008}), and the inequality~\eqref{RelationStabilityGeometryMagnitude} becomes trivial at the limit. 

We point out that assertion $2.$ may not always improve assertion $1.$ In dimension $n=2$, we have that $\int_{\D\Omega}\gamma=2\pi>0$ from the Gauss-Bonnet theorem, and we can show that $a\mapsto \mu_a$ reaches (a unique) positive maximum at some $a_\star>0$ (see \cite[Theorem~2.1]{Umezu} with $g\equiv0$, $h\equiv-\gamma$). However, we cannot guarantee that $a_\star\geq1$ in general.

Besides providing criteria for the non-existence of patterns, inequality~\eqref{RelationStabilityGeometryMagnitude} gives an upper bound on the stability of patterns in terms of the magnitude of the nonlinearity and the geometry of the domain.

%
%
%
%
%
%
%
%
%

\paragraph*{}
As a consequence of~\autoref{th:CoroInstablePRELIMINARY}, we derive  the non-existence of patterns if the domain is shrinking or if the nonlinearity is vanishing, when $\mu_\gamma>0$.
\begin{corollary}\label{Coro_Shrinking}
Fix $(\Omega,f)$ such that $\mu_\gamma>0$, $\sup f'<+\infty$, and set $0<\eta\leq \sqrt{\frac{\mu_\gamma}{\sup f'}}$. Then, there exists no (possibly unstable) solutions to~\eqref{ANG_Intro_EquationSemilineaire} for $(\eta\Omega, f)$ or $(\Omega,\eta^2  f)$, or more generally for any $(\eta^{1-\alpha}\Omega,\eta^{2\alpha} f)$, with $\alpha\in(0,1)$.
\end{corollary}
The proof of this result relies on the scaling property of $\mu_\gamma$. For $\eta>0$, denote $\gamma_\eta(\cdot):=\eta^{-1}\gamma(\eta^{-1}\cdot)$ the minimal curvature of $\eta\Omega$ from~\autoref{Def:CurvatureGamma}. Then, we have
$$\mu_{\gamma_\eta}(\eta\Omega)=\eta^{-2}\mu_\gamma(\Omega),$$
and so applying~\autoref{th:CoroInstablePRELIMINARY} proves the above corollary.

The above corollary can be put into perspective with a theorem of Dancer that implies the non-existence of stable patterns in \emph{dilated domains}, see \autoref{th:Dancer} in Section~\ref{sec:Literature}.
\paragraph*{}
To apply our results to even more concrete situations, we need lower bounds on $\mu_\gamma$ in terms of classical geometry quantities (in-radius, curvature, etc). If the domain is convex, such a bound can be inferred from the literature on the Laplacian associated with standard positive Robin boundary condition (e.g.~\cite{Kovarik2012,Savo2019,Umezu,Bareket1975}). For example, \cite[Corrolary~3]{Savo2019} yields
\begin{equation}\label{Mu_LowerBound}
\mu_{a\gamma}>\frac{\pi^2a\underline{\gamma}}{4R^2a\underline{\gamma}+\pi^2R},\qquad \forall a\geq 1,
\end{equation}
where $R$ is the in-radius of $\Omega$ and $\underline{\gamma}=\inf_{\D\Omega}\gamma$.
Owing to this lower bound and~\autoref{th:CoroInstablePRELIMINARY}, we can improve~\autoref{ANG_th:Intro_CHM} in the following way.
\begin{corollary}\label{th:ApplicationConvex}
Assume that $\Omega$ is a convex domain of in-radius $R$ and lowest curvature $\underline{\gamma}:=\inf_{\D\Omega}\gamma$.
\begin{enumerate}
\item If $\sup f'<\frac{\pi^2\underline{\gamma}}{4\underline{\gamma}R^2+\pi^2R}$, there exists no (possibly unstable) pattern.
\item If $\frac{\pi^2\underline{\gamma}}{4\underline{\gamma}R^2+\pi^2R}\leq \sup f'$, then, for $u$ a pattern of~\eqref{ANG_Intro_EquationSemilineaire},
\begin{equation}\label{EstimateExampleConvex}
\lambda_0< \sup f' -\frac{2\pi^2\underline\gamma}{8R^2\underline{\gamma}+\pi^2R}.
\end{equation}
\end{enumerate}
\end{corollary}
To prove this corollary, simply plug~\eqref{Mu_LowerBound} into~\eqref{RelationStabilityGeometryMagnitude}, take $a=1$ for the first assertion and $a=2$ for the second assertion. One can also improve~\eqref{EstimateExampleConvex} by taking the best value of $a$ in~\eqref{RelationStabilityGeometryMagnitude}.

The above result indeed improves the conclusion of~\autoref{ANG_th:Intro_CHM} if $\sup f'$ is not too large: the first assertion states the non-existence of all (possibly unstable) patterns if $\sup f'$ is small enough, and the second assertion the non-existence of some unstable patterns if $\sup f'$ is intermediate (note that the right member of~\eqref{EstimateExampleConvex} is negative if $\sup f'$ is close to $\frac{\pi^2\underline{\gamma}}{4\underline{\gamma}R^2+\pi^2R}$). However, if $\sup f'$ is large, then the right member~\eqref{EstimateExampleConvex} is positive while \autoref{ANG_th:Intro_CHM} already implies $\lambda_0<0$.

\subsection{Nonlinear Cacciopoli inequality}

The following result is a general gradient estimate for the solutions of~\eqref{ANG_Intro_EquationSemilineaire} which can be seen as a \emph{nonlinear Caccioppoli inequality}, or an \emph{inverse Poincaré inquality}. 
For the sake of generality, the result is stated for solutions of heterogeneous semi-linear equations, namely, when the nonlinearity $f=f(x,u)$ also depends explicitely on $x\in\Omega$.
\begin{theorem}\label{th:CoroStable}
Let $u$ be a solution of
\begin{equation}\label{EquationSemilineaire_Hetero}
    \left\{
    \begin{aligned}
        &-\Delta u(x)=f(x,u(x)), &&\forall x\in\Omega,\\ 
        &\D_\nu u(x)=0, &&\forall x\in\D\Omega,\\
    \end{aligned}
    \right.
\end{equation}
where $f\in C^{1}\left(\Omega\times\R,\R\right)$, and recall $\mu_\gamma$ from~\autoref{Def:CurvatureGamma}. Then
\begin{equation}\label{CacciopolyInequality}
\mu_{\gamma}\int_\Omega\vert\nabla u\vert^2\leq \int_\Omega f(x,u(x))^2dx,
\end{equation}
with equality if and only if $u$ is constant.
\end{theorem}
\noindent We give a slight refinement of this result in \autoref{Cacciopoli_MorePrecise} (Section~\ref{sec:Cacciopoli}).

We emphasize that the theorem holds for possibly unstable solutions, and is universal in the sense that $\mu_\gamma$ does not depend on $f$. However, it is relevant only if the domain is such that the Robin Laplacian is positive (i.e. $\mu_\gamma\geq0$). This class of domains contains the convex domains.

Multiplying~\eqref{ANG_Intro_EquationSemilineaire} by $u$ and using the divergence theorem, we obtain the identity $\int_\Omega \vert\nabla u\vert^2=\int_\Omega uf(x,u)$, which allows to rewrite~\eqref{CacciopolyInequality} as
\begin{equation}
\int_\Omega f(x,u)\big( f(x,u)-\mu_\gamma u\big)\geq 0.
\end{equation}
Note also that a simple integration of~\eqref{ANG_Intro_EquationSemilineaire} gives $\int_\Omega f(x,u)=0$, and thus $\int_\Omega f(x,u)^2$ is a $L^2$ measure of the variation of $f(x,u)$ around its means value.

As mentioned before, the inequality~\eqref{CacciopolyInequality} can be seen as an inverse Poincar\'e inequality. Incidentally, it yields the upper bound
\begin{equation}\label{UpperBoundMuPoincare}
\mu_\gamma< C_{P},
\end{equation}
where $C_P$ is the best constant for the second Poincar\'e inequality (or, equivalently, the inverse of the spectral gap of the Laplacian with Neumann boundary conditions). To prove this, simply take $f(x,u)=u$ and $u$ the eigenfunction associated with $C_P$ in~\eqref{CacciopolyInequality}. This inequality implies that $\mu_\gamma$ is typically very small or negative in domains featuring a narrow bottleneck, such as Matano's dumbbell domains, introduced in Section~\ref{sec:Literature} below.

\subsection{Flatness estimates}

The literature conveys the idea that stable patterns tend to be flat (i.e. one-dimensional). This indeed holds in domains of the form $\R\times\omega$ or $\R^2\times\omega$ with $\omega\subset\R^{n-2}$ convex and bounded~\cite{Nordmann2019a} (but false in $\R^n$ when $n\geq8$~\cite{Pacard2013}). Let us provide quantitative estimates on the flatness of patterns.

To begin with, let us refine \autoref{th:CoroStable} and infer a flatness estimate in terms of the spectral gap of the Robin-curvature Laplacian. The following result deals with the non-homogeneous problem~\eqref{EquationSemilineaire_Hetero}. While \autoref{th:CoroStable} involves the lowest eigenvalue $\mu_\gamma$ of the Robin-curvature Laplacian, the following flatness estimate involves the \emph{second lowest eigenvalue} $\mu_{\gamma,2}$, defined as
\begin{equation}\label{ANG_Intro_DefLambdaGamma2}
\mu_{\gamma,2}:=\sup\Big\{\inf\{\mathcal{G}_\gamma(\psi),\psi\in E,                                                                                                                                                                                                                                                                                                                                                                                                          \Vert \psi\Vert_{{L}^2}=1\}, E\subset H^1(\Omega), \dim E=2\big\},
\end{equation}
with $\mathcal{G}_\gamma$ from \autoref{Def:CurvatureGamma}.
We point out that, since $\mu_\gamma$ is associated with a unique (up to renormalization) eigenfunction $\varphi$~\cite{Daners2013}, an equivalent definition is given by
\begin{equation}
\mu_{\gamma,2}:=\inf\left\{ \mathcal{G}_\gamma(\psi), \psi\in H^1(\Omega), \Vert \psi\Vert_{{L}^2}=1, \int_\Omega\psi\varphi=0\right\}.
\end{equation}
The fact that $\mu_{\gamma,2}-\mu_\gamma>0$ is referred to in the literature as the \emph{spectral gap} (in our context, we may call it the \emph{Robin-curvature spectral gap}). Many estimates on the spectral gap for Neumann boundary conditions are available~\cite{Mufa1997,Chen1997,Andrews2011,Rohleder2014}, and some of them can be adapted to the Robin-curvature boundary conditions using the method of~\cite{Daners2013}.

\begin{proposition}\label{th:FlatnessEstimateRobinCurvatureCoro}
Let $u$ be a pattern of the non-homogeneous equation~\eqref{EquationSemilineaire_Hetero}. There exists a direction $e\in\mathbb{S}^{n-1}$ such that,
\begin{equation}
(\mu_{\gamma,2}-\mu_{\gamma}) \frac{\Vert \nabla u\cdot e\Vert^2}{\Vert \nabla u\Vert^2}\geq \mu_{\gamma,2}-\sup f',
\end{equation}
where $\Vert\cdot\Vert$ denotes the usual $L^2$-norm.
\end{proposition}
This proposition implies that patterns tend to be flat when $\mu_{\gamma,2}-\mu_{\gamma}\gg1$.
An analogous $k$-dimensional estimate can be derived with the $k$-ieth lowest eigenvalue of the Robin-curvature Lapacian, see \autoref{Remark_Flatness_Robin} (Section~\ref{sec:flatnessEstimates}).
We also point out that the estimate in \autoref{th:FlatnessEstimateRobinCurvatureCoro} holds if we replace $\sup f'$ by $\frac{\int_\Omega f(u)^2}{\Vert \nabla u\Vert^2}$, see~\autoref{rmk_Flatness_intF}.

\paragraph*{}
The following result is another flatness estimate for patterns of Morse index~$1$. For $u$ a solution of~\eqref{ANG_Intro_EquationSemilineaire}, we define $\lambda_{\gamma,2}$ the \emph{second lowest eigenvalue} of $\mathcal{F}_\gamma$ (\autoref{Definition_Lambda_Gamma}) by
\begin{equation}
\lambda_{\gamma,2}:=\sup\Big\{\inf\{\mathcal{F}_\gamma(\psi),\psi\in E,                                                                                                                                                                                                                                                                                                                                                                                                          \Vert \psi\Vert_{{L}^2}=1\}, E\subset H^1(\Omega), \dim E=2\big\},
\end{equation}
Accordingly, we define $\la_{0,2}$ by replacing $\mathcal{F}_\gamma$ with $\mathcal{F}_0$ (\autoref{def:stability}) in the above expression. 
We usually say that a solution of~\eqref{ANG_Intro_EquationSemilineaire} is of \emph{Morse index $1$} if $\lambda_0<0\leq \lambda_{0,2}$, that is, if there is only one direction of perturbation for which $u$ is unstable. This class of solutions is of particular importance, for example when considering solutions obtained from a mountain pass procedure or a minimization under constraint.
Analgously, we say that a solutionis of \emph{Robin-curvature Morse index}$1$ if $\lambda_\gamma<0\leq \lambda_{\gamma,2}$.

The following estimate is relevant for pattern of Morse index or Robin-curvature Morse index $1$.
\begin{proposition}\label{th:FlatnessEstimateRobinCurvature}
Let $u$ be a pattern of~\eqref{ANG_Intro_EquationSemilineaire}. There exists a direction $e\in\mathbb{S}^{n-1}$ such that
\begin{equation}\label{FlatnessEstimateRobinCurvature}
1\geq \frac{\Vert \nabla u\cdot e\Vert^2}{\Vert \nabla u\Vert^2}\geq \frac{\lambda_{\gamma,2}}{\lambda_{\gamma,2}-\lambda_\gamma},
\end{equation}
where $\Vert\cdot\Vert$ denotes the usual $L^2$-norm.

In addition, there exists a direction $e\in\mathbb{S}^{n-1}$ such that, for any $a\geq 1$,
\begin{equation}\label{FlatnessEstimateRobinCurvatureCoro2}
(a-1)(\lambda_{0,2}-\lambda_0)\frac{\Vert \nabla u\cdot e\Vert^2}{\Vert \nabla u\Vert^2}\geq (a-1)\lambda_{0,2}+\mu_{a\gamma}-\sup f'.
\end{equation}
\end{proposition}
This proposition is a refinement of \autoref{ANG_Intro_th:GammaStability}.
Indeed, if $u$ is \emph{Robin-curvature-stable} (i.e. $\lambda_\gamma\geq0$), then~\eqref{FlatnessEstimateRobinCurvature} implies that $\frac{\lambda_{\gamma,2}}{\lambda_{\gamma,2}-\lambda_\gamma}=1$, and that $\nabla u=\vert\nabla u\vert e$, hence $u$ is one-dimensional; it is then easy to deduce that $u$ is constant. 

Estimate~\eqref{FlatnessEstimateRobinCurvatureCoro2} implies that patterns tend to be flat when $\lambda_{0,2}-\lambda_0\gg1$. An analogous $k$-dimensional estimate can be derived for patterns of Morse index $k\leq n$, see \autoref{Remark_Flatness_Robin} (Section~\ref{sec:flatnessEstimates}).

\paragraph*{}
Several refinements of the above flatness estimates are proposed in a series of remarks in Section~\ref{sec:flatnessEstimates}.

We also give another flatness estimate in \autoref{th:FlatnessEstimate_Poincare} (Section~\ref{sec:Proof_2}) which relies on the geometric Poincaré inequality introduced by Sternberg and Zunbrum~\cite{Sternberg1998a,Sternberg1998}.

\subsection{Unbounded domains}

Let us now study whether our results extend to unbounded domains. We consider our main equation
\begin{equation}\label{EquationSemilineaire_Unbounded}
    \left\{
    \begin{aligned}
        &-\Delta u(x)=f(u(x)), &&\forall x\in\Omega,\\ 
        &\D_\nu u(x)=0, &&\forall x\in\D\Omega,\\
        &u\in  L^\infty\cap C^{2} \left(\overline{\Omega}\right),
    \end{aligned}
    \right.
\end{equation}
where the domain $\Omega\subset\R^n$ ($n\geq2$) is possibly unbounded and uniformly $C^2$, and, as before, $f$ is $C^{1}$.

The classification of solutions in unbounded domains is a very active topic that is related to many fundamental problems.
A vast literature is devoted to the classification of stable patterns in the particular case of the entire space $\Omega=\R^n$ and the Allen-Cahn nonlinearity $f(u)=u-u^3$: it is known that stable patterns are necessarily planar in dimensions $n=1,2$~\cite{Berestycki1997b} and that it is not true in dimensions $n\geq8$~\cite{Pacard2013}. The question remains open for the intermediate dimensions $3\leq n\leq 7$. These results are closely related to De Giorgi's conjecture. We refer the reader to~\cite{Wei2018} for a state of the art on this question.

In a previous paper~\cite{Nordmann2019a}, we studied whether~\autoref{ANG_th:Intro_CHM} extends to unbounded domains. 
We showed that, if the domain is convex and possibly unbounded, there exist no patterns to~\eqref{EquationSemilineaire_Unbounded} which are stable non-degenerate (i.e., with $\lambda_0>0$). 
We also gave a classification of possibly degenerate stable solutions (i.e., with $\lambda_0\geq0$) if the domain is convex and further satisfies the growth condition at infinity
\begin{equation}\label{GrowthCondition}
\vert \Omega\cap \{\vert x\vert\leq R\}\vert= O(R^2),\qquad \text{when }R\to+\infty.
\end{equation}
More precisely, under assumption~\eqref{GrowthCondition}, if the domain is convex and is not a straight cylinder (i.e., the domain is not of the form $\Omega=\R\times\omega$, with $\omega\subset\R^{n-1}$), then there exists no patterns to~\eqref{EquationSemilineaire_Unbounded}. If the domain is a convex straight cylinder, stable solutions are either constant or monotonic planar solutions connecting two stable roots $(z_1,z_2)$ of $f$ such that $\int_{z_1}^{z_2}f=0$.

In particular, this extends the one-dimensional symmetry of stable patterns in $\R^2$ to any convex domain satisfying~\eqref{GrowthCondition}.

We point out that none of the above holds for unbounded solutions: indeed, $u(x):=e^x$ is a solution of $-u''=-u$ in $\R$ for which $\lambda_0>0$. It justifies why we only consider bounded solutions in~\eqref{EquationSemilineaire_Unbounded}.

\paragraph*{}
The present article improves the results of~\cite{Nordmann2019a} and thus extends some results on De Giorgi's conjecture (namely the  one-dimensional symmetry of stable patterns in $\R^2$) to a larger class of domains.

In view of~\autoref{ANG_Intro_th:GammaStability}, it is reasonnable to expect that the results of~\cite{Nordmann2019a}, which hold for stable solutions in convex domains, remain true for \emph{Robin-curvature-stable solutions} in any (possibly non-convex) domain. This is what we state in~\autoref{th:Unbounded_Domains} below.

We notice that for $u$ a solution of~\eqref{EquationSemilineaire_Unbounded}, the definition of Robin-curvature-stability through the sign of $\lambda_\gamma$ in~\autoref{Definition_Lambda_Gamma} remains licit. By convention, we set $\gamma=0$ if $\Omega=\R^n$.
We also define $\lambda_\gamma$ on the truncated domains
\begin{equation}\label{DefLambda_R}
\la_\gamma^R:=\inf\limits_{\substack{\psi\in H^1(\Omega_R)\cap H^1_0(B_R)\\ \Vert \psi\Vert_{{L}^2}=1}} \mathcal{F}_\gamma(\psi),\qquad \forall R>0,
\end{equation}
where $\Omega_R:=\Omega\cap B_R$ and $B_R$ is the ball of radius $R$ centered at the origin.
Here, the infimum is taken over functions $\psi\in H^1(\Omega_R)$ whose support is included in $B_R$. The fact that we impose Dirichlet boundary conditions on $\Omega\cap\D B_R$ ensures that $R\mapsto \lambda_R$ is decreasing. 
It is proved in~\cite[Theorem~2.1]{Rossi2020} that
\begin{equation}
\lim_{R\to+\infty}\lambda_\gamma^R=\lambda_\gamma.
\end{equation}

The following result extends~\autoref{ANG_Intro_th:GammaStability} to unbounded domains under either a non-degeneracy assumption on the Robin-curvature-stability, or a control on the growth of the domain at infinity.
\begin{theorem}\label{th:Unbounded_Domains}
Let $u$ be a solution of~\eqref{EquationSemilineaire_Unbounded} and assume that $u$ is Robin-curvature-stable (i.e. $\lambda_\gamma\geq0$).
\begin{enumerate}
\item If the Robin-curvature-stability of $u$ is not too degenerate, i.e.,
\begin{equation}\label{Robin_Curvature_stability_not_too_degenerate}
\liminf\limits_{R\to+\infty} R^2\lambda_\gamma^R=+\infty
\end{equation}
(in particular, if $\lambda_\gamma>0$), then $u$ is constant.
\item If $\Omega$ satisfies the growth condition at infinity~\eqref{GrowthCondition} and is not a straight cylinder (i.e. $\Omega$ is not of the form $\R\times \omega$ with $\omega\subset\R^{n-1}$), then $u$ is constant.
\item If $\Omega=\R\times\omega$ satisfies the growth condition at infinity~\eqref{GrowthCondition}, then $u$ is either constant, or is a monotonic planar solution connecting two stable roots $(z_1,z_2)$ of $f$ such that $\int_{z_1}^{z_2}f=0$.
\end{enumerate}
\end{theorem}
We recall that in a convex domain, owing to~\eqref{MonotonieLambdaGammaPositif}, stability implies Robin-curvature-stability. The above theorem therefore contains the results of~\cite{Nordmann2019a} as special cases.

Assumption~\eqref{Robin_Curvature_stability_not_too_degenerate} in the first assertion of~\autoref{th:Unbounded_Domains} encompasses the case of nondegenerate Robin-curvature-stability $\lambda_\gamma>0$. Therefore, this result improves~\cite[Theorem~1.4]{Nordmann2019a} which only deals with the case of nondegenerate stability ($\lambda_0>0$).

The second and third assertions of the above theorem give a complete classification of Robin-curvature-stable patterns in domains satisfying~\eqref{GrowthCondition}. This is a refinement of~\cite[Theorem~1.5]{Nordmann2019a} which deals with stable solutions in convex domains. We point out that assumption~\eqref{GrowthCondition} is satisfied by any subset of $\R^2$ and also by some subsets of $\R^n$, $n>2$, which size do not grow too fast at infinity.

If the domain is a straight cylinder, then $\gamma\leq0$, and so $\lambda_\gamma\leq\lambda_0$: in this case, Robin-curvature-stability implies stability. If the domain is a convex straight cylinder, then $\lambda_\gamma=\lambda_0$ and the third assertion in  \autoref{th:Unbounded_Domains} coincides with~\cite[Theorem~1.5~\textit{2.}]{Nordmann2019a}. 
Planar patterns might exist in such domains. For example, the Allen-Cahn equation in $\R$, $-u''=u(1-u^2)$
admits the explicit solution 
$u : x\mapsto \tanh\frac{x}{\sqrt{2}}$ which is stable degenerate (i.e. $\lambda_\gamma=\lambda_0=0$).

\paragraph*{}

The restriction~\eqref{Robin_Curvature_stability_not_too_degenerate} on the growth of the domain at infinity is believed to be sharp, see~\cite{Nordmann2019a} for more details. In~\cite{Nordmann2019a} we asked whether an assumption of \emph{strict convexity} of the domain allows to relax~\eqref{Robin_Curvature_stability_not_too_degenerate} for the non-existence of stable patterns. 
Unfortunately, \autoref{th:Unbounded_Domains} does not provide a positive answer to this question, because we may not have that $\lambda_\gamma>\lambda_0$ in general, even if the domain is strictly convex.

Let give more details.
For $\alpha\geq 0$, consider the principal eigenvalue of the $\alpha$-Robin Laplacian
\begin{equation}\label{def:mu_alpha}
\mu_\alpha:= \inf\limits_{\substack{\psi\in H^1(\Omega)\\                                                                                                                                                                                                                                                                                                                                                                                                          \Vert \psi\Vert_{{L}^2}=1}} \int_\Omega \vert \nabla \psi\vert^2+\alpha\int_{\D\Omega}\psi^2.
\end{equation}
If the domain is bounded, it is classical that $\alpha\mapsto\mu_\alpha$ is strictly increasing (see e.g. the last part of the proof of~\autoref{th:MonotonieStricteLambdaGammaPositif}). 
If the domain is unbounded, we still have that $\alpha\mapsto\mu_\alpha$ is nondecreasing, however, the strict monotonicity does not hold in general. Indeed, for $\Omega=(0,+\infty)$, we have $\mu_\alpha=0$ for all $\alpha\geq0$. To prove this, notice that the function $\psi(x)=\alpha x+1$ is positive on $[0,+\infty)$ and satisfies $\Delta \varphi=0$ and $-\varphi'(0)+\alpha\varphi(0)=0$; hence $\varphi$ is a positive supersolution of the Robin Laplacian in $\Omega$, which classically implies $\mu_\alpha\leq0$.

The fact that $\alpha\mapsto \mu_\alpha$ is not strictly monotonic suggests that, even if we assume that the domain is strictly convex (i.e. $\gamma>0$), we cannot guarantee that the strict inequality $\lambda_\gamma>\lambda_0$ holds in general, contrarily to what is stated in \autoref{th:MonotonieStricteLambdaGammaPositif} for bounded domains. Hence, even in a strictly convex domain ($\gamma>0$), a stable degenerate pattern ($\lambda_0=0$) can have a degenerate Robin-curvature-stability (i.e. $\lambda_\gamma=0$). 

\paragraph*{}
The following result states that~\autoref{th:CoroInstablePRELIMINARY} holds in unbounded domain.
\begin{corollary}\label{th:CoroInstablePRELIMINARY_Unbounded}
Let $u$ be a pattern of~\eqref{EquationSemilineaire_Unbounded}. For any $a\geq 1$, we have that
\begin{equation}\label{RelationStabilityGeometryMagnitude_Unbounded}
(a-1)\lambda_0\leq \sup f'-\mu_{a\gamma},
\end{equation}
with $\mu_a$ from~\autoref{Def:CurvatureGamma}.
Consequently,
\begin{enumerate}
\item If $\mu_\gamma> \sup f'$, there exists no (possibly unstable) pattern to~\eqref{EquationSemilineaire_Unbounded}.
\item If $\mu_{a\gamma}> \sup f'$ for some $a\geq1$, there exists no stable pattern to~\eqref{EquationSemilineaire_Unbounded}.
\end{enumerate}
\end{corollary}
This gives a criterion for the non-existence of patterns in unbounded domains, involving the stability of the solution, the geometry of the domain, and the magnitude of the non-linearity.

We emphasize that~\eqref{RelationStabilityGeometryMagnitude_Unbounded} features a large inequality, whereas~\eqref{RelationStabilityGeometryMagnitude}  features a strict inequality. Accordingly, the condition for the non-existence of pattern in the above theorem involves a large inequality, contrarily to~\autoref{th:CoroInstablePRELIMINARY}.
In fact, sufficient conditions for the non-existence of patterns are
\begin{equation}
\liminf\limits_{R\to+\infty}R^2\left(\mu_\gamma- \sup f'\right)>0,
\end{equation}
or
\begin{equation}
\mu_\gamma\geq \sup f'\quad\text{and}\quad\Omega\text{ satisfies }~\eqref{GrowthCondition}.
\end{equation}

Let us now adapt~\autoref{th:CoroStable} to unbounded domains. Note that  $\nabla u$ and $f(x,u(x))$ may not belong to $L^2$.
We thus have to consider a cut-off function:
for $R>0$ and $\delta>0$, set
\begin{equation}\label{DefCutOff}
\chi_R(x):=\chi\left(\frac{\vert x\vert}{R}\right), \quad \forall x\in\R^n,
\end{equation} 
with $\chi$ a smooth nonnegative function such that $${\chi(z)=
\left\{\begin{aligned}
&1 &&\text{if }0\leq z\leq 1\\
&0 &&\text{if } z\geq 2
\end{aligned}\right.},\qquad\qquad\vert \chi'\vert\leq 2.$$
Let us also define $\mu_\gamma^R$ the principal eigenvalue of the Robin-curvature Laplacien on the truncated domain $\Omega_R:=\Omega\cap B_R$ as 
\begin{equation}\label{DefMu_R}
\mu_\gamma^R:=\inf\limits_{\substack{\psi\in H^1(\Omega_R)\cap H^1_0(B_R)\\ \Vert \psi\Vert_{{L}^2}=1}} \mathcal{G}_\gamma(\psi),\qquad \forall R>0.
\end{equation}

\begin{theorem}\label{th:CoroStable_Unbounded}
Let $u$ be a solution of
\begin{equation}\label{EquationSemilineaire_Hetero_Unbounded}
    \left\{
    \begin{aligned}
        &-\Delta u(x)=f(x,u(x)), &&\forall x\in\Omega,\\ 
        &\D_\nu u(x)=0, &&\forall x\in\D\Omega,\\
        &u\in  L^\infty\cap C^{2} \left(\overline{\Omega}\right),
    \end{aligned}
    \right.
\end{equation}
where $f\in C^{1}\left(\Omega\times\R,\R\right)$ and $\Omega$ is not necessarily bounded. Then
\begin{equation}\label{CacciopolyInequality_Unbounded}
\mu_\gamma^R \int_\Omega\chi_R^2\vert\nabla u\vert^2\leq \int_\Omega\big(\vert \nabla \chi_R\vert \vert\nabla u\vert + \chi_Rf(x,u)\big)^2 dx.
\end{equation}
\end{theorem}
Compared to~\autoref{th:CoroStable}, the right hand-side of~\eqref{CacciopolyInequality_Unbounded} feature an error term $\vert \nabla \chi_R\vert \vert\nabla u\vert$ which is due to the truncation.
The following corollary states that this term can somehow be neglected when $R\to+\infty$.
\begin{corollary}\label{th:Coro_Estimate_Unbounded_Limsup}
Under the same conditions, if $u$ is not constant, then
\begin{equation}
\limsup\limits_{R\to+\infty} \frac{\int_\Omega \chi_R^2 f(x,u)^2 dx}{\int_\Omega\chi_R^2\vert\nabla u\vert^2} \geq \frac{\mu_\gamma}{2}.
\end{equation}
\end{corollary}

The flatness estimates stated in \autoref{th:FlatnessEstimateRobinCurvatureCoro} and \autoref{th:FlatnessEstimateRobinCurvature} formally hold in unbounded domains. However, there is no spectral gap in unbounded domains (not even for the standard Laplacian in $\R$) and therefore these results are not relevant.

\section{Related literature}\label{sec:Literature}
In this section, we present the main known results in the literature on the existence, non-existence, and qualitative properties of patterns. We also give more details on the implications of our results.

\paragraph*{Existence of stable patterns in dumbbell domains.}
In his pioneering article~\cite{Matano1979a}, Matano constructs a counterexample to Theorem~\ref{ANG_th:Intro_CHM} by proving that~\eqref{ANG_Intro_EquationSemilineaire} admits stable patterns for some $(\Omega, f)$.
The author considers a generic \emph{bistable} nonlinearity, and a so-called \emph{dumbbell domain}, which consists in two disjointed convex sets connected by a bottleneck of width ${\eps\ll 1}$, see \autoref{ANG_IntG_fig:Dumbbell}. He proves the existence of a stable pattern if $\eps$ is small enough (and provides a quantitative estimate on the smallness of $\eps$).
\begin{figure}[!h]
    \centering\includegraphics[width=0.5\textwidth]{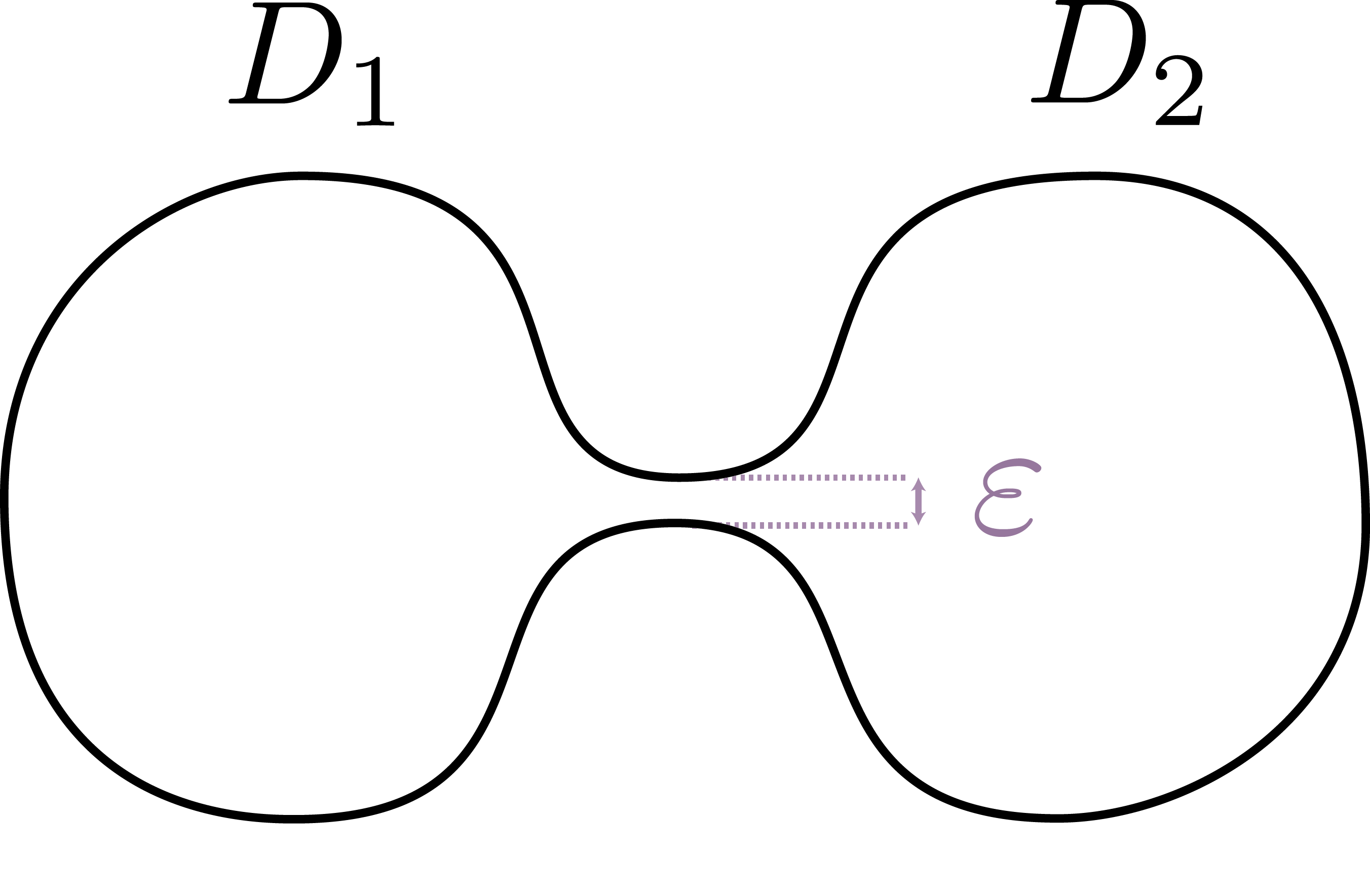}
    \caption{dumbbell domain\label{ANG_IntG_fig:Dumbbell}}
\end{figure}
This construction has been simplified and extended in many directions~\cite{Hale1984,Consul1999,Gokieli2009,Bandle2012,Jimbo1984}, and a general theoretical framework to deal with dumbbell domains is provided in~\cite{Arrieta2006,Arrieta2009,Arrieta2009a}.

The best constant in the second Poincar\'e inequality, denoted by $C_P$, somehow measures the connectivity of the domain ($C_P$ is also the inverse of the spectral gap of the Laplacian with Neumann boundary conditions).
Heuristically, in dumbbells domains, the presence of the bottleneck hinders the connectivity of the domain and so the second Poincar\'e inequality is typically very weak, i.e., $C_P\ll1$. Inequality~\eqref{UpperBoundMuPoincare} implies that $\mu_\gamma$ is also typically very small. We point out that \autoref{th:CoroInstablePRELIMINARY} contains a somehow converse statement since it guarantees the non-existence of patterns if $\mu_\gamma$ is large enough (compared to the magnitude of the nonlinearity).

We emphasize that the existence of stable patterns relies more on the presence of a bottleneck (which must be very narrow compared to $f$) than on the non-convexity of the domain. This can be seen in dimension $n=2$ by the following result, obtained after a simple rescaling from a result of Dancer.
\begin{theorem}[Theorem~5 in \cite{Dancer2004}]\label{th:Dancer}
Let $\Omega\subset\R^2$ be a smooth domain, and assume $0\in\Omega$.
Further assume that $f$ satisfies
\begin{equation}\label{Assumptionfbounded}
\left\{\begin{gathered}
 \exists (M,m)\in\R^2\text{ such that }
 \begin{cases}
 f>0 &on $(-\infty,m)$,\\
 f<0 &on $(M,+\infty)$.
 \end{cases}\\
 \text{the stable roots of $f$, denoted $(z_i)$ are simple and isolated},\\
 \forall z_i\neq z_j,\quad \int_{z_i}^{z_j}f\neq 0.
 \end{gathered}\right.
 \end{equation} 
Then, there exists no stable pattern to~\eqref{ANG_Intro_EquationSemilineaire} in the dilated domain $\eps^{-1}\Omega$, if $\eps\ll1$.
\end{theorem}
This theorem can be put into perspective with \autoref{Coro_Shrinking} which state the non-existence of patterns in shrinking domains.

\paragraph*{Link with minimal surfaces.}

Consider equation~\eqref{ANG_Intro_EquationSemilineaire} with the rescaled Allen-Cahn nonlinearity $f_\eps(u):=\frac{1}{\eps} (u-u^3)$.  This equation arises in the Van der Waal-Cahn-Hiliard model for phase transitions~\cite{Cahn1958}.
A series of pioneering papers~\cite{Modica1977,Modica1979,Caffarelli2006,Savin2009} establishes that when $\eps\to0$,  any level set of a sequence of stable solutions $u_\eps$ converges to a set $E$ which is a local minimizer of the perimeter functional in $\Omega$ (in other words, $\D E\cap \Omega$ is a minimal surface in $\Omega$). 

Conversely, Kohn and Sternberg~\cite{Kohn1989} show that, given such a set $E$ with a minimal perimeter in $\Omega$, there exists a sequence of stable solutions $u_\eps$ converges to $\mathds{1}_E-\mathds{1}_{\Omega\backslash E}$.
Thus, provided there exists such a (non-trivial) set $E$, it proves that there exists a stable pattern to~\eqref{ANG_Intro_EquationSemilineaire} for $f_\eps$, $\eps\ll1$. 
A typical example is the dumbbell domain, with $\D E$ located at the bottleneck of the domain (\autoref{fig:DumbbellMinimalSurface}); thus Kohn and Stenberg's pattern echoes Matano's. 
On the contrary, if the domain $\Omega$ is a convex domain, there exists no such set $E$; accordingly, from Theorem~\ref{ANG_th:Intro_CHM}, there exists no stable pattern as well. 

As mentioned in the introduction, the result of Kohn and Sternberg allows constructing stable patterns in star-shaped domains obtained as a small perturbation of convex domains, see \autoref{ANG_IntG_fig:SurfaceMinimale}. 
In dimension $n\geq 3$, such a domain can also be chosen with positive mean curvature.
In that sense, the convexity assumption in \autoref{ANG_th:Intro_CHM} is sharp.

Nevertheless, we see from~\autoref{th:CoroPerturbation} that, when the domain is close from being convex, this construction can only be achieved when $\eps\ll1$, i.e., when the nonlinearity has a large magnitude.

We also emphasize that, contrary to Matano's construction, Kohn and Sternberg's pattern relies on the fact that the Allen-Cahn nonlinearity is balanced (i.e. $\int_{-1}^{+1} f=0$). This can be seen in dimension $n=2$ by~\autoref{th:Dancer} which, after a simple rescaling, implies the non-existence of patterns for unbalanced nonlinearity of very large magnitude.


\begin{figure}[!h]
    \centering
\begin{subfigure}[b]{0.49\linewidth}
    \centering\includegraphics[width=\textwidth]{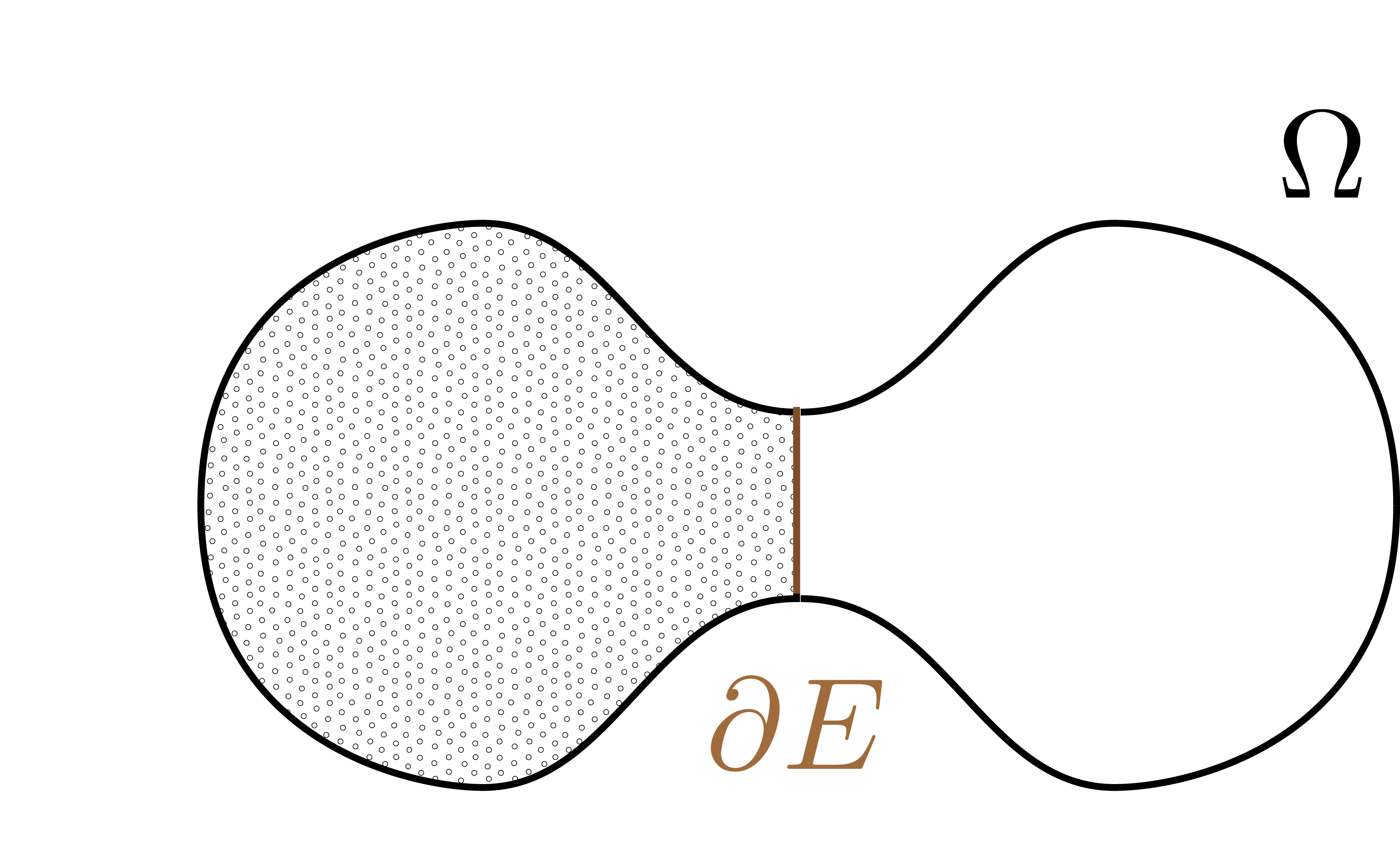}
    \caption{dumbbell domain}\label{fig:DumbbellMinimalSurface}
  \end{subfigure}
  \hfill
    \begin{subfigure}[b]{0.49\linewidth}
    \centering\includegraphics[width=\textwidth]{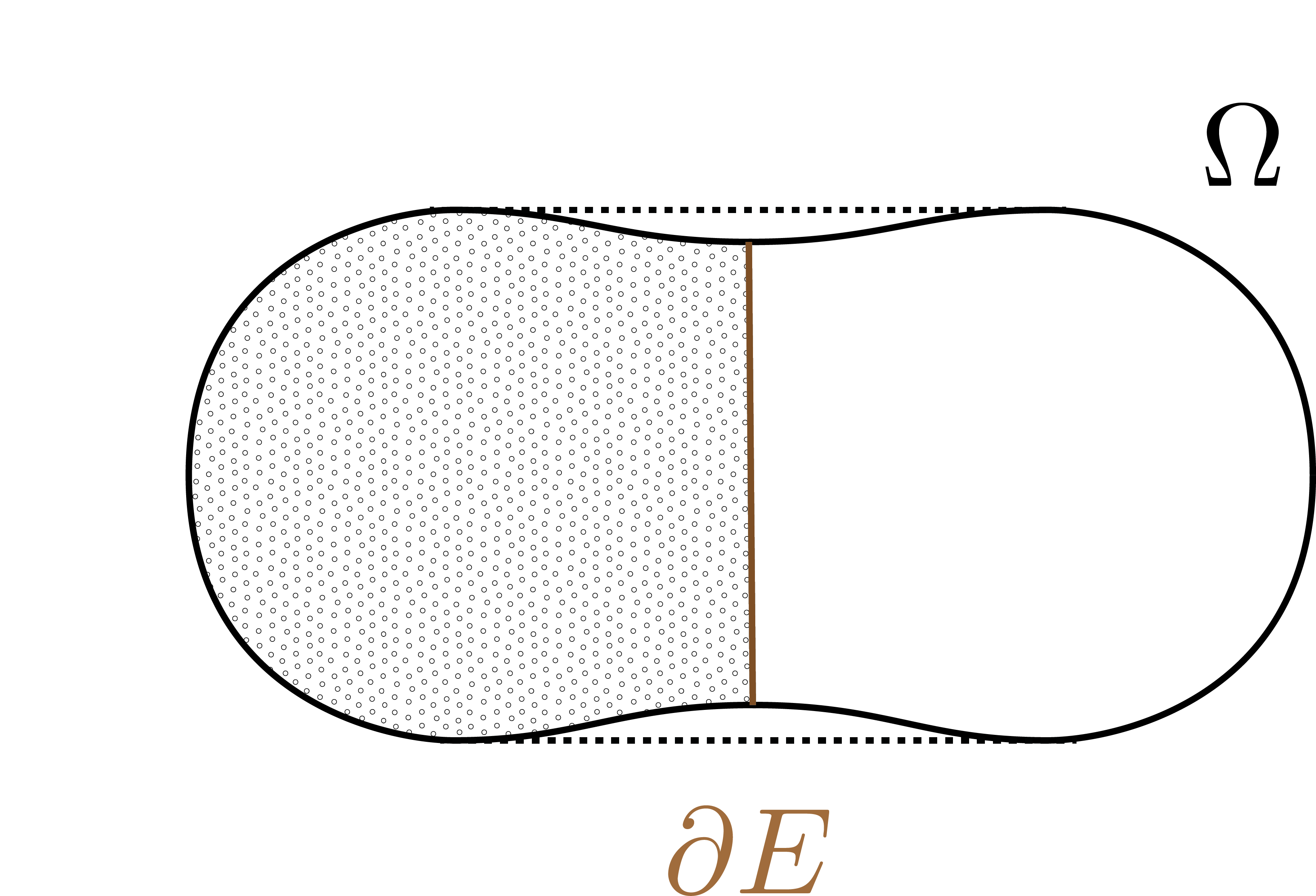}
    \caption{star-shaped domain obtained as a small perturbation of a convex domain}\label{ANG_IntG_fig:SurfaceMinimale}
  \end{subfigure}
  \newline
      \begin{subfigure}[b]{0.4\linewidth}
    \centering\includegraphics[width=\textwidth]{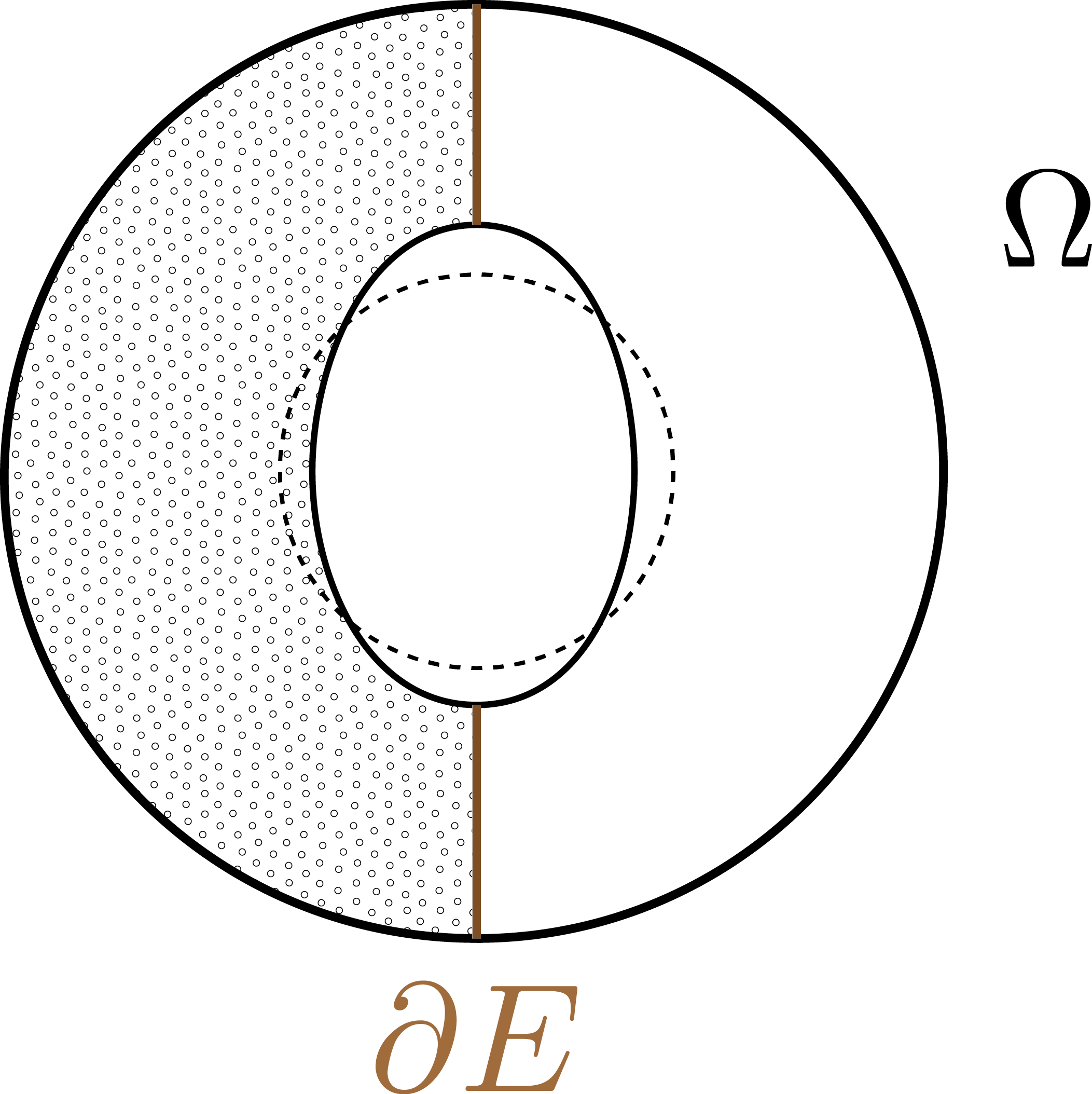}
    \caption{Small perturbation of an annulus}\label{ANG_IntG_fig:Torus}
  \end{subfigure}
  \hspace{1cm}
    \caption{Domains $\Omega$ admitting patterns for $f=\frac{1}{\eps^2}(u-u^3)$, $\eps\ll1$. The dotted area $E$ stands for a set with minimal perimeter in~$\Omega$.}
\end{figure}

\paragraph*{Angular symmetry.}
Matano~\cite{Matano1979a} also shows that, if the domain is a solid of revolution, stable patterns must be invariant with respect to rotations around the axis of revolution. If, in addition, the section is convex, there exists no stable pattern. In particular, it proves the non-existence of stable pattern for any $C^1$ nonlinearity in certain non-convex domains, such as torus or rings. 

However, it has to be noticed that the non-existence of stable pattern is not robust under small perturbation of the domain since patterns can be constructed with Kohn and Sternberg methods in domains as in \autoref{ANG_IntG_fig:Torus}. 

Matano's angular symmetry suggests that stable patterns inherit some symmetries of the domain. It is also true for invariance by translations: stable patterns in $\R\times\omega$ (where $\omega\subset\R^{n-1}$ is bounded) are invariant with respect to $x_1$~\cite{Nordmann2019a}.
Nevertheless, the symmetries by rotation and translation are the only ones to be expected, since Cartesian and angular differentiations are the only first-order differential operators commuting with the Laplacian.
It is, however, natural to consider this question for patterns on manifolds. This has been done by many authors~\cite{Jimbo1984,Rubinstein1994,Bandle2012,Punzo2013,Farina2013a,Farina2013b,Goncalves2010}.

\paragraph*{Propagation in reaction-diffusion equations. }

Many papers deal with the spreading properties of reaction-diffusion equations, which are the evolution equation associated with~\eqref{ANG_Intro_EquationSemilineaire}. In this context, the existence of stable patterns corresponds to the possibility of the blocage of an invading wave; on the other hand, the non-existence of stable patterns means that there is a complete invasion.

Let us briefly report on the article~\cite{Berestyckib} which studies the propagation in the presence of an obstacle.
Consider an exterior domain $\Omega=\R^n\setminus K$, where the \emph{obstacle} $K\subset\R^n$ is compact, and a bistable unbalanced nonlinearity, say $f(u)=u(1-u)(u-a)$, with $a<1/2$, so that $\int_0^1 f>0$. This last inequality (which is not satisfied by the Allen-Cahn nonlinearity) ensures that any traveling wave connecting $0$ to $1$ converges to $1$ in large times (roughly speaking ``$1$ is more stable than $0$''). The authors consider the special class of solution to~\eqref{ANG_Intro_EquationSemilineaire} which satisfies
\begin{equation}\label{Limit1}
\lim\limits_{\vert x\vert\to+\infty}u(x)=1.
\end{equation}
This class of solution is obtained as the limit in large times of a \emph{generalized traveling wave}.

On the one hand, if the obstacle features a bottleneck (as in a dumbbell domain), there exist stable patterns satisfying~\eqref{Limit1}. On the other hand, if the obstacle is star-shaped or axially convex, there exists no (possibly unstable) pattern satisfying~\eqref{Limit1}. We point out, however, that in such domains (see~\autoref{ANG_IntG_fig:Obstacle_Convex2}) there might exist stable patterns which do not satisfy~\eqref{Limit1} but satisfy $\lim_{\vert x\vert\to+\infty}u(x)=0$.
See also~\cite{Berestycki2015e} for similar conclusions if the domain is a cylinder with varying cross section, and~\cite{Brasseur2019,Brasseur2018} for the case of non-local diffusion.

\begin{figure}[!h]
    \centering\includegraphics[width=0.5\textwidth]{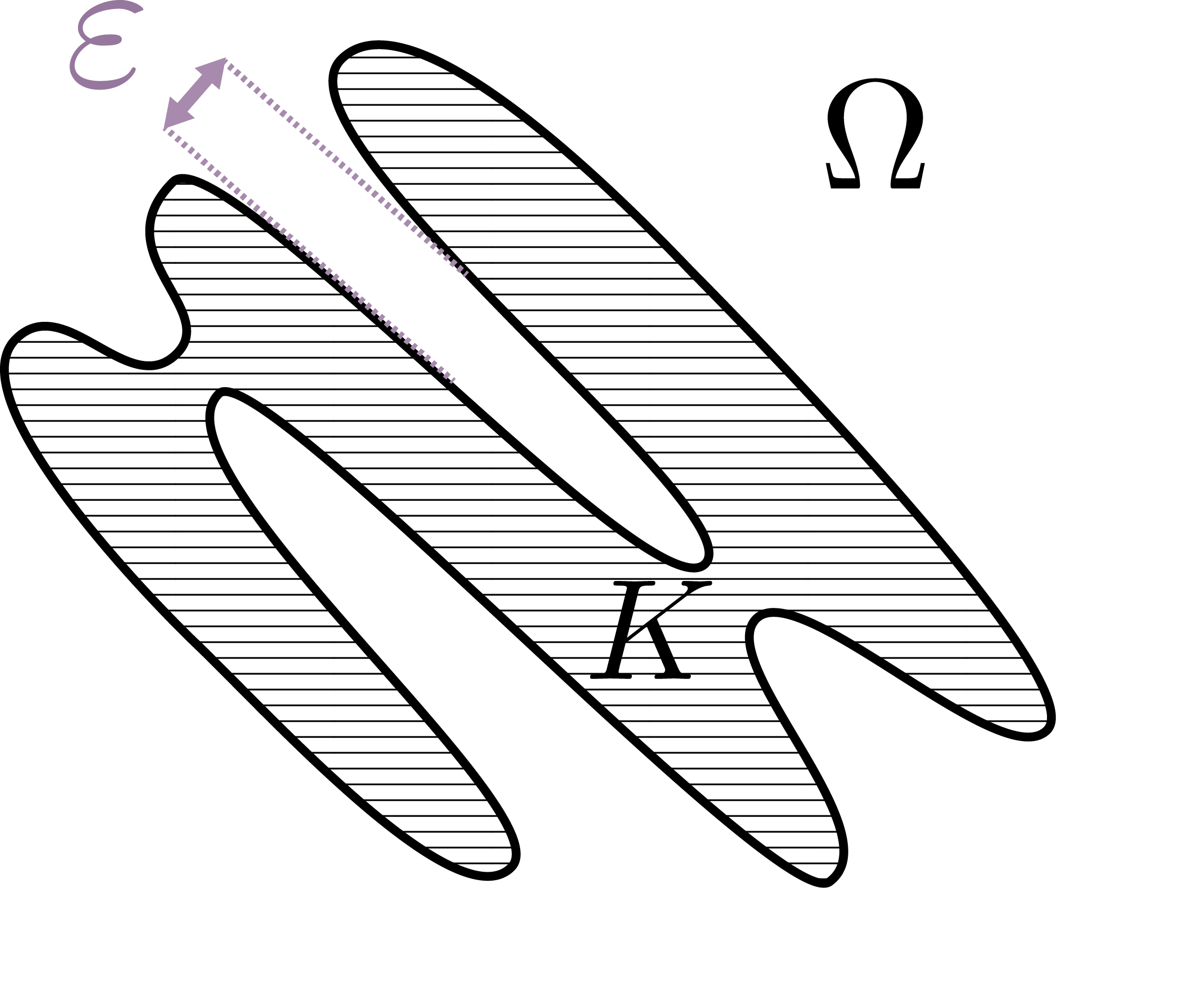}
  \caption{Example of a domain $\Omega=\R^n\setminus K$, exterior of an axially convex obstacle, which does not admit any pattern satisfying~\eqref{Limit1}, but admits stable patterns (for some $f$). \label{ANG_IntG_fig:Obstacle_Convex2}}
\end{figure}

\paragraph*{Further references.}

Many authors have also been interested in understanding how the existence of stable patterns could emerge from the non-homogeneity of coefficients, for example, under non-homogeneous diffusivity~\cite{Yanagida1982,DoNascimento2014a,Sonego2017} or non-homogeneous reaction~\cite{Alikakos988,Brown1990,Sonego2016}. 

Finally, we mention that the regularity of weak stable solutions has been widely studied,  see~\cite{Dupaigne2011,Cabreb}, and references therein. See also~\cite{Chanillo1998a} for properties on the level sets of stable patterns in the case of Dirichlet boundary conditions.

\section{Non-existence of patterns}\label{sec:Proofs_Non-existence}
\subsection{Proof of \autoref{ANG_Intro_th:GammaStability}}\label{sec:Proof_1}

We give a proof which is in the spirit of~\cite{Casten1978a,Matano1979a}; we propose an alternative proof in Section~\ref{sec:Proof_2}.
We begin with a classical geometrical lemma which was first stated in \cite{Casten1978a,Matano1979a} (in a slightly less general form).
\begin{lemma}[\cite{Casten1978a,Matano1979a}] \label{LemmeIntermediaireVariations}
Let $\Omega\subset\R^n$ be a smooth domain and recall $\gamma(\cdot)$ from~\autoref{Def:CurvatureGamma}. Let $u$ be a $C^2$ function such that
\begin{equation}\label{NeumannVariations}
\D_\nu u=0\quad \text{on }\D\Omega.
\end{equation}
Then,
\begin{equation}
\frac{1}{2}\D_\nu\vert\nabla u\vert^2=-\nabla u\cdot\nabla\nu\cdot\nabla u\leq - \gamma \vert\nabla u\vert^2\quad\text{on }\D\Omega,
\end{equation}
with equality if $n=2$.
\end{lemma}
\begin{proof}
On the one hand, differentiating~\eqref{NeumannVariations} with respect to the vector field $\nabla u$ (note that $\nabla u$ is tangential to $\D\Omega$) leads to
\begin{equation}\label{Lemma_Geom_Proof_Intermediate}
0=\nabla\D_\nu u\cdot\nabla u=\nabla u\cdot\nabla^2 u\cdot\nu+\nabla u\cdot\nabla\nu\cdot\nabla u.
\end{equation}
On the other hand, we have 
\begin{equation}
\D_\nu\vert\nabla u\vert^2=\nabla\left(\vert\nabla u\vert^2\right)\cdot\nu=2\nabla u\cdot\nabla^2 u\cdot\nu.
\end{equation}
Using~\eqref{Lemma_Geom_Proof_Intermediate} in the above expression, we deduce
\begin{equation}\label{LemmeIntermediaireEgaliteIntermediaire}
\D_\nu\vert\nabla u\vert^2=-2\nabla u\cdot\nabla\nu\cdot\nabla u.
\end{equation}

We recall that ${\gamma}(x)$ is defined as the lowest eigenvalue of $\nu(x)$ restricted to the tangent space $\nu(x)^\perp$, i.e., by the Rayleigh quotient formula,
\begin{equation}
{\gamma}(x)=\inf\limits_{X\in\nu(x)^\perp\setminus\{0\}}\frac{X\cdot\nabla\nu(x)\cdot X}{\vert X\vert^2}, \qquad\forall x\in\D\Omega.
\end{equation}
Therefore, $\nabla u\cdot\nabla\nu\cdot\nabla u\geq \gamma\vert\nabla u\vert^2$, with equality if $n=2$. From this and~\eqref{LemmeIntermediaireEgaliteIntermediaire}, we conclude the proof.
\end{proof}

Now, we give the following intermediate result.
\begin{lemma}\label{Lemma_InegaliteSommeF}
Let $u$ be a solution of~\eqref{ANG_Intro_EquationSemilineaire}, and set $v_i:=\D_{x_i} u$, for all $i\in\{1,\dots,n\}$. We have that
\begin{equation}\label{InegaliteSommeF}
\sum\limits_{i=1}^n \mathcal{F}_\gamma(v_i)\leq 0,
\end{equation}
with $\mathcal{F}_\gamma$ from~\autoref{Definition_Lambda_Gamma}.
\end{lemma}
\begin{proof}
Differentiating \eqref{ANG_Intro_EquationSemilineaire} with respect to $x_i$, we find that $v_i:=\D_{x_i}u$ satisfies the linearized equation
\begin{equation}\label{linearization}
    -\Delta v_i - f'(u)v_i=0\quad\text{in }\Omega.
\end{equation}
Multiplying by $v_i$ and using the divergence theorem, we find
\begin{equation*}
\int_\Omega \vert \nabla v_i\vert ^2 - \int_\Omega f'(u)v_i^2=\int_{\D\Omega}v_i\D_\nu v_i,
\end{equation*}
and so, writing $v_i\D_\nu v_i= \frac{1}{2}\D_\nu v_i^2$, we obtain
\begin{equation*}
\mathcal{F}_\gamma(v_i)=\int_{\D\Omega}\frac{1}{2}\D_\nu v_i^2+\int_{\D\Omega}\gamma v_i^2.
\end{equation*}
Summing the above equation for $i\in\{1,\dots,n\}$ and using that $\vert \nabla u\vert^2=v_1^2+\dots+v_n^2$ gives
\begin{equation}
\sum\limits_{i=1}^n \mathcal{F}_\gamma(v_i)=\int_{\D\Omega}\frac{1}{2}\D_\nu \vert\nabla u\vert^2+\int_{\D\Omega}\gamma \vert\nabla u\vert^2.
\end{equation}
The right-hand side of the above identity is nonpositive from \autoref{LemmeIntermediaireVariations}, which achieves the proof.
\end{proof}

With these lemmas in hands, we are ready to prove~\autoref{ANG_Intro_th:GammaStability}.
Let $u$ be a solution of \eqref{ANG_Intro_EquationSemilineaire} and assume $\lambda_\gamma\geq0$. Set $v_i:=\D_{x_i} u$, for all $i\in\{1,\dots,n\}$.
Since $\lambda_\gamma\geq0$, we have that $\mathcal{F}_\gamma(\cdot)\geq0$, hence, using~\autoref{Lemma_InegaliteSommeF},
\begin{equation}
0\leq \mathcal{F}_\gamma(v_i)\leq\sum\limits_{k=1}^n \mathcal{F}_\gamma(v_k)=\int_{\D\Omega}\frac{1}{2}\D_\nu \vert\nabla u\vert^2+\int_{\D\Omega}\gamma \vert\nabla u\vert^2\leq 0.
\end{equation}
We infer that $\mathcal{F}_\gamma(v_i)=0$ and that $v_i$ minimizes $\mathcal{F}_\gamma$.

It is then classical that $v_i$ must have a constant strict sign on $\overline\Omega$ (more precisely, $v_i$ must be a constant multiple of the eigenfunction associated with $\lambda_\gamma$~\cite{Daners2013}).
Then, from $\D_\nu u=0$ on the closed surface $\D\Omega$, we deduce that $v_i$ vanishes on some point of the boundary. Since $v_i$ has a constant strict sign, we deduce that $v_i\equiv0$, which completes the proof.

\subsection{Alternative proof of \autoref{ANG_Intro_th:GammaStability}}\label{sec:Proof_2}
The proof of \autoref{ANG_Intro_th:GammaStability} presented in Section~\ref{sec:Proof_1} relies on \autoref{Lemma_InegaliteSommeF} which involves the term $\sum_{i=1}^n \mathcal{F}_\gamma(v_i)$. However, it is not easy to give an intuitive interpretation of this quantity, and although the method is simple, it seems ad hoc. 
Let us give an alternative proof where the above quantity appears naturally. The idea is inspired by the famous geometric Poincaré inequality introduced by Sternberg and Zunbrum~\cite{Sternberg1998a,Sternberg1998} and consists of using $v=\vert \nabla u\vert$ as a test function in $\mathcal{F}_\gamma$.

In fact, we are going to prove a more precise result than \autoref{ANG_Intro_th:GammaStability}.
\begin{proposition}\label{th:FlatnessEstimate_Poincare}
Let $u$ be a nonconstant solution of~\eqref{ANG_Intro_EquationSemilineaire}. Then,
\begin{equation}\label{SommeFGeometricPoincare_Statement}
\lambda_\gamma \leq  \frac{\big\Vert\nabla \vert \nabla u\vert \big\Vert^2-\Vert \nabla^2 u\Vert^2}{ \Vert \nabla u\Vert^2}< 0,
\end{equation}
where $\Vert\cdot\Vert$ denotes the usual $L^2$-norm and $\nabla^2 u$ is the Hessian matrix of $u$.
\end{proposition}
\begin{proof}
Set $v_i=\D_{x_i} u$ for all $i=1,\dots, n$ and $v=\vert \nabla u\vert =\sqrt{v_1^2+\dots+v_n^2}$.
A straightforward computation gives (on the set $\{v\neq0\}$)
\begin{equation}
-\Delta v - f'(u) v= \frac{1}{v}\left(\vert\nabla v\vert^2-\sum_{i=0}^n\vert \nabla v_i\vert^2\right)\leq0.
\end{equation}
Multiplying by $v$ and using the divergence theorem, we find
\begin{equation}
-\int_{\D\Omega} v\D_\nu v+ \int_\Omega \vert\nabla v\vert^2- \int_\Omega  f'(u)v^2=\int_\Omega \bigg[\vert\nabla v\vert^2-\sum_{i=0}^n\vert \nabla v_i\vert^2\bigg].
\end{equation}
Writing $v\D_\nu v=\frac{1}{2}\D_\nu v^2$, and using \autoref{LemmeIntermediaireVariations}, we obtain
\begin{equation}\label{SommeFGeometricPoincare}
\mathcal{F}_\gamma(v)\leq \int_\Omega \bigg[ \vert\nabla v\vert^2-\sum_{i=0}^n\vert \nabla v_i\vert^2\bigg].
\end{equation}
From the definition of $\lambda_\gamma$ (\autoref{Definition_Lambda_Gamma}), we have that $\mathcal{F}_\gamma(v)\geq \lambda_\gamma\Vert v\Vert^2$.
We conclude the proof by using this inequality and that $\int_\Omega\sum_{i=0}^n\vert \nabla v_i\vert^2=\Vert\nabla^2 u\Vert$ in~\eqref{SommeFGeometricPoincare}.
\end{proof}

We emphasize that~\eqref{SommeFGeometricPoincare} is simply a rewriting of~\eqref{InegaliteSommeF} (we recall that $v^2= v_1^2+\dots+ v_n^2$). The advantage of this formulation is that we can give a a geometric interpretation of the right member: as stated in~\cite[Lemma 2.1]{Sternberg1998},
\begin{equation}
\vert\nabla v\vert^2-\sum_{i=0}^n\vert \nabla v_i\vert^2=
\left\{\begin{aligned}
&\vert \nabla u\vert^2\sum_{i=1}^{n-1}\kappa_i^2+\big\vert\nabla_T\vert \nabla u\vert\big\vert^2 &&\text{on }\{\nabla u\neq 0\},\\
&0 &&\text{a.e. on }\{\nabla u=0\},
\end{aligned}\right.
\end{equation}
where $\kappa_i$ are the principal curvature of the level set of $u$ and $\nabla_T$ denotes the differentiation which is tangent to this level set. 

In particular, this term is nonpositive and is identically zero if and only if $u$ is one-dimensional. Thus, inequality~\eqref{SommeFGeometricPoincare_Statement} can be seen as a flatness estimate.

If $\lambda_\gamma\geq0$, we deduce from~\eqref{SommeFGeometricPoincare} and the above geometric interpretation that $u$ is one-dimmensional. Then, we easily conlcude that $u$ is constant: it completes the proof of \autoref{ANG_Intro_th:GammaStability}.

For $u$ a pattern of~\eqref{ANG_Intro_EquationSemilineaire}, inequality~\eqref{SommeFGeometricPoincare_Statement} provides a negative upper bound on $\lambda_\gamma$.
\autoref{th:FlatnessEstimate_Poincare} is therefore more precise than~\autoref{ANG_Intro_th:GammaStability} which only implies that $\lambda_\gamma<0$.

\subsection{Proof of~\autoref{th:CoroPerturbation}}

We begin with an intermediate result stating that the strict inequality holds in~\eqref{MonotonieLambdaGammaPositif}.
\begin{lemma}\label{th:MonotonieStricteLambdaGammaPositif}
Assume that $\Omega$ is bounded and convex, and let $u$ be a solution of~\eqref{ANG_Intro_EquationSemilineaire}. Then,
\begin{equation}\label{MonotonieStricteLambdaGammaPositif}
\lambda_\gamma> \lambda_0
\end{equation}
with $\lambda_0$ $\lambda_\gamma$ from \autoref{def:stability} and \autoref{Definition_Lambda_Gamma}.
\end{lemma}
\begin{proof}
First, let us prove that there exists a point $x_0\in\D\Omega$ at which $\gamma(x_0)>0$. Since $\Omega$ is bounded, the norm mapping $x\in\Omega\mapsto \vert x\vert$ reaches its maximum $R>0$ at some point $x_0\in\D\Omega$. Denoting by $B_{R}$ the ball of radius $R$ centered at the origin, we have that $\Omega\subset B_{R}$ and that $\Omega$ and $B_R$ are tangent at $x_0$. From a standard comparison principle, we deduce that $\gamma(x_0)\geq\frac{1}{R_0}>0$.

Then, consider $\varphi$ the (unique) positive eigenfunction associated with $\lambda_\gamma$ (given, for example, by~\cite{Daners2013}). It is classical that $\varphi>0$ in $\overline{\Omega}$. From this and the fact that $\gamma(x_0)>0$, we infer that $\int_{\D\Omega}\gamma\varphi^2>0$. We deduce
\begin{equation}
\lambda_\gamma=\mathcal{F}_\gamma(\varphi)= \mathcal{F}(\varphi)+\int_{\D\Omega}\gamma\varphi^2>\mathcal{F}_0(\varphi)\geq \lambda_0,
\end{equation}
which concludes the proof.
\end{proof}

Now, let us prove~\autoref{th:CoroPerturbation}.
Assume that $\Omega$ is convex and let $(\Omega_\eps,f_\eps)$, $\eps>0$ be a family of smooth perturbations which converges to $(\Omega,f)$ as $\eps\to0$. By contradiction, assume that there exists $u_\eps$ a sequence of stable patterns of~\eqref{ANG_Intro_EquationSemilineaire} with $(\Omega_\eps,f_\eps)$. We denote $\lambda_{0,\eps}$ and $\lambda_{\gamma,\eps}$ the quantities from \eqref{DefLambda} and \eqref{ANG_Intro_DefLambdaGamma} corresponding to $u_\eps$.
Since $u_\eps$ is a stable pattern, we have $\lambda_{0,\eps}\geq0$ (by the definition of stability) and $\lambda_{\gamma,\eps}<0$ (from \autoref{ANG_Intro_th:GammaStability}), which implies
\begin{equation}\label{IntermediateInequalityEigen}
\lambda_{\gamma,\eps}<\lambda_{0,\eps}.
\end{equation}

From classical elliptic estimates, $u_\eps$ is bounded in $C^{2,\alpha}(\overline\Omega_\eps),$ uniformly in $\eps>0$ (see e.g.~\cite[Theorem 6.30]{GilbarDavid2015}). Up to extraction of a subsequence (still denoted $\eps$), $u_\eps$ converges when $\eps\to0$ to some $u$ which is a solution of~\eqref{ANG_Intro_EquationSemilineaire}. Identically, $\lambda_{0,\eps}$ and $\lambda_{\gamma,\eps}$ converge to some $\lambda_0$, $\lambda_\gamma$ respectively.

Since $\lambda_{0,\eps}$ and $\lambda_{\gamma,\eps}$ behave continuously as $\eps\to0$, inequality~\eqref{IntermediateInequalityEigen} implies $\lambda_\gamma\leq0\leq \lambda_0$: we get a contradiction with~\autoref{th:MonotonieStricteLambdaGammaPositif}.

%
%

\subsection{Proof of \autoref{th:CoroInstablePRELIMINARY}}\label{sec:CriteriaNonExistence}

The proof of \autoref{th:CoroInstablePRELIMINARY} follows directly from the following lemma which breaks down the quantity $\lambda_\gamma$ into three parts: the stability of the solution, the geometry of the domain (through $\mu_\gamma$ from \autoref{Def:CurvatureGamma}), and the magnitude of the nonlinearity.
\begin{lemma}\label{th:Lemma_lambdaGammaInequalityBreakDown}
Let $u$ be a solution of~\eqref{ANG_Intro_EquationSemilineaire} and let $a\geq 1$. Then,
\begin{equation}
a \lambda_\gamma \geq (a-1)\lambda_0+\mu_{a\gamma}- \sup f'.
\end{equation}
\end{lemma}
If $u$ is not constant, then \autoref{ANG_Intro_th:GammaStability} implies that $\lambda_\gamma<0$. Then, the proof of \autoref{th:CoroInstablePRELIMINARY} is directly deduced from the above lemma.
\begin{proof}
Fix $b \in(0,1)$ and $\psi\in H^1(\Omega)$. 
By writing $\vert\nabla \psi\vert^2= b \vert\nabla\psi\vert^2+(1- b )\vert\nabla \psi\vert^2$, a straightforward computation gives
\begin{align}
&\mathcal{F}_\gamma(\psi):=\int_\Omega\vert\nabla \psi\vert^2 - \int_\Omega f'(u)+\int_{\D\Omega}\gamma\psi^2\\
&= b \left( \int_\Omega\vert\nabla \psi\vert^2- \int_\Omega f'(u) \psi^2\right)+(1- b )\left(\int_\Omega\vert\nabla \psi\vert^2+\int_{\D\Omega}\frac{\gamma }{1-b}\psi^2-\int_\Omega f'(u)\psi^2\right)\\
&= b\mathcal{F}_0(\psi) + (1-b)\mathcal{G}_{\frac{1}{1-b}\gamma}(\psi)- (1-b)\int_\Omega f'(u)\psi^2,\label{Intermediaire_PourAutrePreuve}
\end{align}
with $\mathcal{F}_\gamma$ from \autoref{Definition_Lambda_Gamma}, $\mathcal{F}_0$ from \autoref{def:stability} and $\mathcal{G}_{\frac{1}{1-b}\gamma}$ from~\autoref{Def:CurvatureGamma}. 
By definition, $\mathcal{F}_0(\psi)\geq \lambda_0\int_\Omega \psi^2$ and $\mathcal{G}_{\frac{1}{1-b}\gamma}(\psi)\geq \mu_{\frac{1}{1-b}\gamma}\int_\Omega \psi^2$. 
Dividing by $(1-b)$ and setting $a=\frac{1}{1-b}$, we infer
\begin{equation}
a \mathcal{F}_\gamma(\psi)\geq \Big((a-1)\lambda_0+\mu_{a\gamma}- \sup f'\Big)\int_\Omega\psi^2.
\end{equation}
We conclude the proof by taking the infimum over all functions $\psi\in H^1$ such that $\int_\Omega\psi^2=1$.
\end{proof}


\section{Nonlinear Cacciopoli inequality}\label{sec:Cacciopoli}

\begin{proof}[Proof of \autoref{th:CoroStable}]
Let $u$ be a solution of~\eqref{EquationSemilineaire_Hetero}, and set $v_i:=\D_{x_i} u$, for all $i\in\{1,\dots,n\}$. 
Differentiating \eqref{EquationSemilineaire_Hetero} with respect to $x_i$, we find that $v_i$ satisfies
\begin{equation}\label{linearization_Hetero}
    -\Delta v_i - \D_u f(x,u)v_i- \D_{x_i} f(x,u)=0\quad\text{in }\Omega.
\end{equation}
Multiplying by $v_i$ and using the divergence theorem, we find
\begin{equation*}
-\int_{\D\Omega}v_i\D_\nu v_i+\int_\Omega \vert \nabla v_i\vert ^2= \int_\Omega \D_u f(x,u)v_i^2+\int_\Omega \D_{x_i} f(x,u)v_i.
\end{equation*}
Writing $v_i\D_\nu v_i= \frac{1}{2}\D_\nu v_i^2$, summing the above equation for $i\in\{1,\dots,n\}$ and using that $\vert \nabla u\vert^2=v_1^2+\dots+v_n^2$, we obtain
\begin{equation}\label{Inegalite_Intermediare_Preuve_Hetero}
-\frac{1}{2}\int_{\D\Omega}\D_\nu\left[\vert\nabla u\vert^2\right] +\sum\limits_{i=1}^n\int_\Omega \vert \nabla v_i\vert ^2 =\int_\Omega \D_u f(x,u)\vert\nabla u\vert^2+\int_\Omega \nabla_x f(x,u)\cdot \nabla u.
\end{equation}

Let us estimate separately the left and right hand-side of the above expression. To estimate the left hand-side, we use \autoref{LemmeIntermediaireVariations} to deduce that $\frac{1}{2}\D_\nu\left[\vert\nabla u\vert^2\right]\leq \gamma \vert\nabla u\vert^2=\gamma\sum_{i=1}^n v_i^2$. Then, recalling the definitions of $\mu_\gamma$ and $\mathcal{G}_\gamma$ from \autoref{Def:CurvatureGamma}, we have
\begin{equation}
-\frac{1}{2}\int_{\D\Omega}\D_\nu\left[\vert\nabla u\vert^2\right] +\sum\limits_{i=1}^n\int_\Omega \vert \nabla v_i\vert ^2\geq \sum_{i=1}^n \mathcal{G}_\gamma(v_i)\geq \mu_\gamma \sum_{i=1}^n \int_\Omega v_i^2= \mu_\gamma \int_\Omega\vert\nabla u\vert^2.
\end{equation}

Let us show that the above inequality is strict whenever $u$ is not constant. If there is equality, we have that $\mathcal{G}_\gamma(v_i)= \mu_\gamma \int_\Omega v_i^2$ for all $i\in\{1,\dots,\}$. Then, $v_i$ minimizes $\mathcal{G}_\gamma$ and therefore is a constant multiple of the principal eigenfunction associated to $\mu_\gamma$. In particular, $v_i$ has a constant strict sign on $\overline{\Omega}$. However, since $\D_\nu u=0$ on $\D\Omega$, we deduce that $v_i$ vanishes on some point of $\D\Omega$, hence $v_i\equiv0$, and so $u$ is constant. We have thus shown that
\begin{equation}\label{Inegalite_Intermediare_Preuve_Hetero_2}
-\frac{1}{2}\int_{\D\Omega}\D_\nu\left[\vert\nabla u\vert^2\right] +\sum\limits_{i=1}^n\int_\Omega \vert \nabla v_i\vert ^2
\geq \mu_\gamma \int_\Omega\vert\nabla u\vert^2,
\end{equation}
with equality if and only if $u$ is constant.

Now, we compute the right-hand side of~\eqref{Inegalite_Intermediare_Preuve_Hetero} using the divergence theorem and that $u$ is a solution of~\eqref{EquationSemilineaire_Hetero}:
\begin{align}
\int_\Omega \D_u f(x,u)\vert\nabla u\vert^2+\int_\Omega \nabla_x f(x,u)\cdot \nabla u
&=\int_\Omega \nabla\big[f(x,u(x))\big]\cdot \nabla u
\\
&=\int_{\D\Omega}f(x,u)\D_\nu udx -\int_\Omega f(x,u) \Delta udx
\\
&= \int_\Omega f(x,u)^2dx.
\end{align}
We complete the proof by plugging the above inequality and~\eqref{Inegalite_Intermediare_Preuve_Hetero_2} into~\eqref{Inegalite_Intermediare_Preuve_Hetero}.
\end{proof}

\begin{remark}\label{Cacciopoli_MorePrecise}
With the same method as in Section~\ref{sec:Proof_2}, we can use the geometric Poincaré inequality to prove the slightly more precise estimate: for $u$ a nonconstant solution of~\eqref{EquationSemilineaire_Hetero}, we have
\begin{equation}
\mu_\gamma \Vert \nabla u\Vert^2-\int_\Omega f(x,u)^2dx\leq \big\Vert\nabla \vert \nabla u\vert \big\Vert^2-\Vert \nabla^2 u\Vert^2<0,
\end{equation}
where $\Vert \cdot\Vert$ denotes the usual $L^2$ norm and $\nabla^2 u$ is the Hessian matrix of $u$.
\end{remark}

\section{Flatness estimates}\label{sec:flatnessEstimates}

\subsection{Proof of \autoref{th:FlatnessEstimateRobinCurvature}}
\textit{Proof of the first statement.}
Let $\varphi$ be the (unique up to renormalization) eigenfunction associated with $\lambda_\gamma$~\cite{Daners2013}. Denoting $\langle\cdot,\cdot\rangle$ the usual $L^2(\Omega)$ scalar product, the mapping $\xi\in\R^n\mapsto \langle \nabla u\cdot \xi,\varphi\rangle$ is a continuous linear form which vanishes on a hyperplane $H$. We consider $(e_2,\dots,e_n)$ an orthonormal basis of $H$, that we supplement with a unit vector $e_1\in H^\perp$ (so that $(e_1,e_2,\dots,e_n)$ is an orthonormal basis of $\R^n$).
Then, we set $v_i:= \nabla u \cdot e_i$.

By construction, $v_i$ is orthogonal to $\varphi$ for all $i\in\{2,\dots,n\}$, hence
\begin{equation}\label{Flatness_Inequality_3}
\mathcal{F}_{\gamma}(v_i)\geq
\left\{\begin{aligned}
&\lambda_{\gamma,2}\Vert v_i\Vert^2 &&\text{if }i\in\{2,\dots,n\},\\
&\lambda_{\gamma}\Vert v_1\Vert^2 &&\text{if }i=1.
\end{aligned}\right.
\end{equation}
From \autoref{Lemma_InegaliteSommeF} and that $\Vert \nabla u\Vert^2= \sum_{i=1}^n\Vert v_i\Vert^2$, we have
\begin{equation}
0\geq  -(\lambda_{\gamma,2}-\lambda_\gamma )\Vert v_1\Vert^2+ \lambda_{\gamma,2} \Vert \nabla u\Vert^2.
\end{equation}
We easily conclude from the above inequality.

\paragraph*{}

\textit{Proof of the second statement.}
As in the previous step, we can choose an orthonormal basis such that, for all $i\in\{2,\dots,n\}$, $v_i$ is orthogonal to the eigenfunction associated with $\lambda_{0}$.
We thus have, for all $a\geq 1$,
\begin{equation}\label{Flatness_Inequality_2}
\mathcal{F}_{0}(v_i)\geq
\left\{\begin{aligned}
&\lambda_{0,2}\Vert v_i\Vert^2 &&\text{if }i\in\{2,\dots,n\},\\
&\lambda_{0}\Vert v_1\Vert^2 &&\text{if }i=1,
\end{aligned}\right.,
\qquad\text{and}\qquad 
\mathcal{G}_{a\gamma}(v_i)\geq 
\mu_{a\gamma}\Vert v_i\Vert^2.
\end{equation}
From~\eqref{Intermediaire_PourAutrePreuve}, for all $b\in(0,1)$, we have 
\begin{equation}
0\geq b\sum\limits_{i=1}^n\mathcal{F}_0(v_i) + (1-b)\sum\limits_{i=1}^n\mathcal{G}_{\frac{1}{1-b}\gamma}(v_i)- (1-b)\int_\Omega f'(u)\vert\nabla u\vert^2.
\end{equation}
Using~\eqref{Flatness_Inequality_2}, we obtain
\begin{equation}
0\geq b\left(\lambda_0\Vert v_1\Vert^2+\lambda_{0,2}\sum\limits_{i=2}^n\Vert v_i\Vert^2\right)+ (1-b)\mu_{\frac{1}{1-b}\gamma}\sum\limits_{i=1}^n\Vert v_i\Vert^2- (1-b)\sup f'\Vert\nabla u\Vert^2.
\end{equation}
Since $\sum\limits_{i=1}^n\Vert v_i\Vert^2=\Vert\nabla u\Vert^2$, dividing by $(1-b)\Vert\nabla u\Vert^2$ and setting $a=\frac{1}{1-b}$ gives
\begin{equation}
0\geq (a-1)\left(\lambda_0\Vert v_1\Vert^2+\lambda_{0,2}\left(\Vert \nabla u\Vert^2-\Vert v_1\Vert^2\right)\right)+ \mu_{a\gamma}\Vert \nabla u\Vert^2- \sup f'\Vert\nabla u\Vert^2.
\end{equation}
Dividing by $\Vert\nabla u\Vert^2$ and rearranging the term gives the result.

\begin{remark}\label{Remark_Flatness_Robin}
An analogous $k$-dimensional estimate can be derived for pattern of Robin-curvature Morse index $k\in\{1,\dots, n\}$. More precisely, let $u$ be a pattern and define, for all positive integer $k$,
\begin{equation}\label{ANG_Intro_DefLambdaGamma_k}
\la_{\gamma,k}:=\sup\Big\{\inf\{\mathcal{F}_\gamma(\psi),\psi\in E,                                                                                                                                                                                                                                                                                                                                                                                                          \Vert \psi\Vert_{{L}^2}=1\}, E\subset H^1(\Omega), \dim E=k\big\}.
\end{equation}
Note that with these notations, we have $\lambda_\gamma=\lambda_{\gamma,1}$.
Then, there exists an orthonormal basis $(e_1,\dots,e_n)$ of $\R^n$ such that
\begin{equation}
\sum\limits_{i=1}^n\lambda_{\gamma,i}\Vert \nabla u\cdot e_i\Vert^2\leq 0.
\end{equation}
In particular, if $u$ is of \emph{Robin-curvature Morse index} $ k\in\{1,\dots,n\}$ (i.e. $\lambda_{\gamma,1}\leq\dots\leq \lambda_{\gamma,k-1}<0\leq \lambda_{\gamma,k}$), then
\begin{equation}
\frac{\sum\limits_{i=1}^k\Vert \nabla u\cdot e_i\Vert^2}{\Vert \nabla u\Vert^2}\leq \frac{\lambda_{\gamma,k}}{\lambda_{\gamma,k}-\lambda_{\gamma,1}},
\end{equation}
and $u$ tends to be $k$-dimensional if $\lambda_{\gamma,k}\gg1$.

Similarily, we can derive an estimate if $u$ is of Morse index $k\in\{1,\dots, n\}$. Setting, for all positive integer $k$,
\begin{equation}
\la_{0,k}:=\sup\Big\{\inf\{\mathcal{F}_0(\psi),\psi\in E,                                                                                                                                                                                                                                                                                                                                                                                                          \Vert \psi\Vert_{{L}^2}=1\}, E\subset H^1(\Omega), \dim E=k\big\},
\end{equation}
there exists an orthonormal basis $(e_1,\dots,e_n)$ such that 
\begin{equation}
(a-1)(\lambda_{0,k}-\lambda_0)\frac{\sum\limits_{i=1}^k\Vert \nabla u\cdot e_i\Vert^2}{\Vert \nabla u\Vert^2}\geq (a-1)\lambda_{0,k}+\mu_{a\gamma}-\sup f',\qquad \forall a\geq 1
\end{equation}

\paragraph*{}
The same remark holds for \autoref{th:FlatnessEstimateRobinCurvatureCoro}, i.e., we can derive a $k$-dimensional estimate involving the $k$-ieth eigenvalue of the Robin-curvature Laplacian. More precisely, define
\begin{equation}\label{ANG_Intro_DefMuGamma_k}
\mu_{\gamma,k}:=\sup\Big\{\inf\{\mathcal{G}_\gamma(\psi),\psi\in E,                                                                                                                                                                                                                                                                                                                                                                                                          \Vert \psi\Vert_{{L}^2}=1\}, E\subset H^1(\Omega), \dim E=k\big\},\quad \forall k\in\mathbb{N}.
\end{equation}
Then, there exists an orthonormal basis $(e_1,\dots,e_n)$ of $\R^n$ such that
$
\mathcal{G}(v_i)\geq \mu_{\gamma,i}\Vert v_i\Vert ^2,
$
and so
\begin{equation}
0\geq \sum\limits_{i=1}^n\mu_{\gamma,i}\Vert \nabla u\cdot e_i\Vert^2-\sup f' \Vert \nabla u\Vert^2.
\end{equation}
In particular, for $u$ a solution of~\eqref{ANG_Intro_EquationSemilineaire} we have that
\begin{equation}
\left(\mu_{\gamma,k}-\mu_{\gamma,1}\right)\frac{\sum\limits_{i=1}^k\Vert \nabla u\cdot e_i\Vert^2}{\Vert \nabla u\Vert^2}\leq \mu_{\gamma,k}-\sup f',
\end{equation}
and so $u$ tends to be $k$-dimensional if $\mu_{\gamma,k}\gg1$.
\end{remark}

\begin{remark}\label{rmk_Flatness_intF}
As already mentionned in the proof of \autoref{th:CoroStable}, from an integration by part, we have that $\int_\Omega f'(u)\vert \nabla u\vert^2=\int_\Omega f(u)^2$.
Hence, the second estimate in \autoref{th:FlatnessEstimateRobinCurvature} holds if we replace $\sup f'$ by $\frac{\int_\Omega f(u)^2}{\Vert \nabla u\Vert^2}$. Namely,
There exists a direction $e\in\mathbb{S}^{n-1}$ such that,
\begin{equation}
(\mu_{\gamma,2}-\mu_{\gamma}) \frac{\Vert \nabla u\cdot e\Vert^2}{\Vert \nabla u\Vert^2}\geq \mu_{\gamma,2}-\frac{\int_\Omega f(u)^2}{\Vert \nabla u\Vert^2},
\end{equation}

The same remark holds for the estimate in \autoref{th:FlatnessEstimateRobinCurvatureCoro}, i.e., there exists a direction $e\in\mathbb{S}^{n-1}$ such that, for all $a\geq1$,
\begin{equation}
(a-1)(\lambda_{0,2}-\lambda_0)\frac{\Vert \nabla u\cdot e\Vert^2}{\Vert \nabla u\Vert^2}\geq (a-1)\lambda_{0,2}+\mu_{a\gamma}-\frac{\int_\Omega f(u)^2}{\Vert \nabla u\Vert^2}.
\end{equation}
\end{remark}

\begin{remark}
In fact, a stronger estimate than~\eqref{FlatnessEstimateRobinCurvature} holds, namely, we can replace $\Vert \nabla u\cdot e\Vert^2$ by $\langle \nabla u\cdot e,\varphi\rangle$, with the normalization $\Vert \varphi\Vert =1$.
Indeed, with the same notations, from the decomposition $L^2(\Omega)=\mathrm{Vect}(\varphi)\oplus\mathrm{Vect}(\varphi)^\perp$, we  write $v_1=a_1\varphi+\psi_1$, with $a_1=\langle v_1,\varphi\rangle$ and $\langle\psi_1,\varphi\rangle=0$. It is then remarkable that $\varphi$ and $\psi_1$ are orthogonal for the quadratic form $\mathcal{F}_\gamma$, namely
\begin{equation}
\mathcal{F}_\gamma(v_1)=a_1^2\mathcal{F}_\gamma(\varphi)+\mathcal{F}_\gamma(\psi_1)\geq a_1^2\lambda_\gamma+\lambda_{\gamma,2}\Vert \psi_1\Vert^2=a_1^2\lambda_\gamma+\lambda_{\gamma,2}\left(\Vert v_1\Vert^2-a_1^2\right).
\end{equation}
Using this inequality in the above proof, we end up with
\begin{equation}
\frac{\int_\Omega \varphi \nabla u\cdot e}{\Vert \nabla u\Vert^2}\geq {\frac{\lambda_2^\gamma}{\lambda_2^\gamma-\lambda_1^\gamma}}.
\end{equation}
If $\lambda_{\gamma,2}\gg1$, the above estimate implies that $u$ must be one-dimensionnal and that $\nabla u\cdot e$ must be a constant multiple of $\varphi$.

The same remark holds for the second estimate in \autoref{th:FlatnessEstimateRobinCurvature} and for \autoref{th:FlatnessEstimateRobinCurvatureCoro}.
\end{remark}

\subsection{Proof of \autoref{th:FlatnessEstimateRobinCurvatureCoro}}

As in the proof of \autoref{th:FlatnessEstimateRobinCurvatureCoro}, we can choose an orthonormal basis such that, for all $i\in\{2,\dots,n\}$, $v_i$ is orthogonal to  the eigenfunction associated with $\mu_{\gamma}$.
Hence,
\begin{equation}\label{Flatness_Inequality_1}
\mathcal{G}_{\gamma}(v_i)\geq
\left\{\begin{aligned}
&\mu_{\gamma,2}\Vert v_i\Vert^2 &&\text{if }i\in\{2,\dots,n\},\\
&\mu_{\gamma}\Vert v_1\Vert^2 &&\text{if }i=1.
\end{aligned}\right.
\end{equation}
From~\eqref{Intermediaire_PourAutrePreuve} with $b=0$ and $\psi=v_i$, we have
\begin{equation}
\mathcal{F}_\gamma(v_i)
= \mathcal{G}_{\gamma}(v_i)- \int_\Omega f'(u)v_i^2.
\end{equation}
Summing the above equality for $i=1,\dots,n$ and using \autoref{Lemma_InegaliteSommeF}, we infer
\begin{equation}
0\geq \sum_{i=1}^n\mathcal{G}_{\gamma}(v_i) -  \sum_{i=1}^n\int_\Omega f'(u)v_i^2.
\end{equation}
Using~\eqref{Flatness_Inequality_1} and that $\vert \nabla u\vert^2=v_1^2+\dots+v_n^2$, we obtain
\begin{equation}
0\geq\mu_{\gamma} \Vert v_1\Vert^2 +\mu_{\gamma,2}\left(\Vert \nabla u\Vert^2-\Vert v_1\Vert^2\right)- \sup f'\Vert\nabla u\Vert^2 .
\end{equation}
We easily conclude from the above inequality.

\begin{remark}\label{Remark_Flatness_General_a}
Using the same method as in the proof of \autoref{th:FlatnessEstimateRobinCurvature}, we can show that, for any $a\geq 1$, there exists a direction $e\in\mathbb{S}^{n-1}$ such that
\begin{equation}
(\mu_{a\gamma,2}-\mu_{a\gamma}) \frac{\Vert \nabla u\cdot e\Vert^2}{\Vert \nabla u\Vert^2}\geq (a-1)\lambda_0+\mu_{a\gamma,2}-\sup f'.
\end{equation}
\end{remark}

\section{Unbounded domains}\label{sec:UnboundedDomains}

\subsection{Proof of \autoref{th:Unbounded_Domains}. First statement}

First, note that~\autoref{LemmeIntermediaireVariations} holds in unbounded domains.
However, it is not the case for \autoref{Lemma_InegaliteSommeF} because $v_i$ may not be integrable and thus the computation of $\mathcal{F}(v_i)$ is not licit.
Let us consider the cut-off function defined in~\eqref{DefCutOff}.
We begin with a lemma which adapts \autoref{Lemma_InegaliteSommeF} to unbounded domains.
\begin{lemma}\label{Lemma_1_Unbounded}
Let $u$ be a solution of~\eqref{EquationSemilineaire_Unbounded}, and set $v_i:=\D_{x_i} u$, for all $i\in\{1,\dots,n\}$. We have that
\begin{equation}\label{InegaliteSommeF_unbounded}
\lambda_\gamma^R \int_\Omega\chi_R^2 \vert \nabla u\vert ^2\leq \sum\limits_{i=1}^n \mathcal{F}_\gamma(\chi_R v_i)\leq \int_\Omega\vert\nabla\chi_R\vert^2\vert\nabla u\vert^2,\qquad \forall R>0,
\end{equation}
with $\mathcal{F}_\gamma$ from~\autoref{Definition_Lambda_Gamma}.
\end{lemma}
\begin{proof}
The first inequality is directly deduced from the definition of $\lambda_\gamma^R$ in~\eqref{DefLambda_R} and the fact that $\vert\nabla u\vert^2=v_1^2+\dots+v_n^2$. Let us focus on the second inequality.
Differentiating \eqref{EquationSemilineaire_Unbounded} with respect to $x_i$, we find that $v_i$ satisfies the linearized equation
\begin{equation}\label{linearization_unbounded}
    -\Delta v_i - f'(u)v_i=0\quad\text{in }\Omega.
\end{equation}
Multiplying by $\chi_R^2v_i$, integrating on $\Omega$ and using the divergence theorem, we find
\begin{equation}
 -\int_{\D\Omega}\chi_R^2v_i\D_\nu v_i+\int_\Omega \nabla\left[\chi_R^2 v_i\right]\cdot\nabla v_i-\int_\Omega f'(u)\chi_R^2v^2_i=0.
\end{equation}
Using $\left\vert\nabla\left[\chi_R v_i\right]\right\vert^2 = \nabla\left[\chi_R^2 v_i\right]\cdot\nabla v_i+\vert\nabla\chi_R\vert^2 v_i^2$ and $v_i\D_\nu v_i= \frac{1}{2}\D_\nu v_i^2$, we derive
\begin{equation}
\int_\Omega \left\vert\nabla\left[\chi_R v_i\right]\right\vert^2-\int_\Omega f'(u)\chi_R^2v^2_i= \int_{\D\Omega}\frac{1}{2}\D_\nu v_i^2\chi_R^2+\int_\Omega\vert\nabla\chi_R\vert^2v_i^2,
\end{equation}
and so
\begin{equation}
\mathcal{F}_\gamma(\chi_R v_i)=  \int_{\D\Omega}\left(\frac{1}{2}\D_\nu v_i^2\chi_R^2+\gamma v_i^2\right)+\int_\Omega\vert\nabla\chi_R\vert^2v_i^2
\end{equation}
Summing the above expression for $i\in\{1,\dots,n\}$ and using that $\vert \nabla u\vert^2=v_1^2+\dots+v_n^2$ gives
\begin{equation}
\sum\limits_{i=1}^n \mathcal{F}_\gamma(v_i)=\int_{\D\Omega}\left(\frac{1}{2}\D_\nu \vert\nabla u\vert^2+\gamma \vert\nabla u\vert^2\right)+\int_\Omega\vert\nabla\chi_R\vert^2\vert\nabla u\vert^2.
\end{equation}
The first term in the right-hand side of the above identity is nonpositive from \autoref{LemmeIntermediaireVariations}, which achieves the proof.
\end{proof}

Let us now prove the first assertion of~\autoref{th:Unbounded_Domains}.
Assume that~\eqref{Robin_Curvature_stability_not_too_degenerate} holds, and assume by contradiction that $u$ is not constant.
On the one hand, from the definition of $\lambda_\gamma^R$ in \eqref{DefLambda_R}, we have
\begin{equation}
\sum\limits_{i=1}^n \mathcal{F}_\gamma(\chi_R v_i)\geq \lambda_\gamma^R \int_\Omega\chi_R^2\vert\nabla u\vert ^2\geq \lambda_\gamma^R \int_{\Omega_R}\vert\nabla u\vert ^2.
\end{equation}
On the other hand, we have that 
\begin{equation}
\int_{\Omega_R}\vert\nabla\chi_R\vert^2\vert\nabla u\vert^2\leq \frac{4}{R^2}\int_{\Omega_{2R}}\vert \nabla u\vert ^2.
\end{equation}
Using the two above inequalities in~\eqref{InegaliteSommeF_unbounded}, we deduce
\begin{equation}\label{Contradiction_Preuve_Unbounded}
\frac{1}{4}R^2\lambda^R_\gamma  \leq \frac{\int_{\Omega_{2R}} \vert \nabla u\vert^2}{\int_{\Omega_R} \vert \nabla u\vert^2}.
\end{equation}
Notice that, since $u$ is not constant, the strong maximum principle implies that $u$ cannot be constant on any open set, and therefore $\int_{\Omega_R} \vert \nabla u\vert^2>0$.
Iterating the above inequality, we find, for $j=1,2,\dots$,
\begin{equation}
\left(\frac{R^2\lambda^R_\gamma}{4}\right)^j  \leq \frac{\int_{\Omega_{2^jR}} \vert \nabla u\vert^2}{\int_{\Omega_R} \vert \nabla u\vert^2}.
\end{equation}
Since $\nabla u$ is bounded and not identically zero, there exists a constant $K$ such that
\begin{equation}
\frac{\int_{\Omega_{2^jR}} \vert \nabla u\vert^2}{\int_{\Omega_R} \vert \nabla u\vert^2}\leq K \left(2^jR\right)^n.
\end{equation}
Therefore,
\begin{equation}
\left(\frac{R^2\lambda^R_\gamma}{2^{n+2}}\right)^j \leq KR^n.
\end{equation}
From~\eqref{Robin_Curvature_stability_not_too_degenerate}, we can choose $R>0$ such that $R^2\lambda^R_\gamma>2^{n+2}$; then, we reach a contradiction letting $j\to+\infty$.

\subsection{Proof of \autoref{th:Unbounded_Domains}. Second and third statements.}

Even if the domain is unbounded, $\lambda_\gamma$ is associated with an eigenfunction $\varphi$ which is positive on $\overline\Omega$ and satisfies
\begin{equation}\label{EquationPrincipalEigenfunction}
\left\{
\begin{aligned}
&-\Delta\varphi-f'(u)\varphi=\lambda_\gamma\varphi &&\text{in }\Omega,\\
&\D_\nu\varphi+\gamma\varphi =0 &&\text{on }\D\Omega.
\end{aligned}
\right.
\end{equation}
The existence of $\varphi$ is proved in~\cite{Rossi2020}. The function $\varphi$ is often refered to as a \emph{generalized} principal eigenfunction, because $\varphi$ might not belong to $H^1(\Omega)$.

If the domain is bounded, it is classical that $\lambda_\gamma$ is a simple eigenvalue of~\eqref{EquationPrincipalEigenfunction}. This property also holds in unbounded domains satisfying~\eqref{GrowthCondition}, as a consequence of the following lemma which is a refinement of~\cite[Theorem~1.7]{Berestycki1997b}.
\begin{lemma}\label{schrodinger2} 
Assume that $\Omega\subset\R^n$ satisfies~\eqref{GrowthCondition} and let $u$ be a Robin-curvature-stable solution of~\eqref{EquationSemilineaire_Unbounded}.
If $v$ is smooth, bounded and satisfies
\begin{equation}\label{LinearEquation}
v\left(-\Delta v-f'(u)v\right)\leq 0,\qquad\text{in }\Omega,
\end{equation}
and
\begin{equation}\label{LinearEquation_Bounary}
v\left(\D_\nu v+\gamma v\right)\leq 0,\qquad \text{on }\D\Omega,
\end{equation}
then $v\equiv C\varphi$ for some constant $C$, where $\varphi$ is a principal eigenfunction associated with $\lambda_\gamma$.
\end{lemma}
\begin{proof}
Let us set ${ \sigma=\frac{v}{\varphi}}$ and show that $\sigma$ is constant. Since $v$ satisfies~\eqref{LinearEquation}-\eqref{LinearEquation_Bounary}, a straightforward computation shows that $\sigma$ satisfies
\begin{equation*}
\sigma\varphi\left(\varphi\Delta \sigma+2\nabla\varphi\cdot\nabla\sigma+\sigma\left(\Delta\varphi+f'(u)\varphi\right)\right)\geq0.
\end{equation*}
Using that $\varphi$ satisfies~\eqref{EquationPrincipalEigenfunction} and that $\lambda_\gamma\geq0$, we have
\begin{equation}
\varphi\sigma^2\left(\Delta\varphi+f'(u)\varphi\right)\leq-\lambda_\gamma\varphi^2\sigma^2\leq 0,
\end{equation}
therefore,
$$
\sigma\varphi\left(\varphi\Delta \sigma+2\nabla\varphi\cdot\nabla\sigma\right)\geq0,
$$
which can be rewritten as
$${
\sigma\nabla\cdot(\varphi^2\nabla\sigma)\geq0.
}$$
Multiplying by $\chi_{R}^2$ (defined in \eqref{DefCutOff}), integrating on $\Omega$ and using the divergence theorem, we find
\begin{align*}
0&\leq \int_{\D\Omega}\chi_{R}^2\sigma\varphi^2\D_\nu\sigma-\int_\Omega\varphi^2\nabla\left(\chi_{R}^2\sigma\right)\cdot\nabla\sigma\\
&=\int_{\D\Omega}\chi_{R}^2\sigma\varphi^2\D_\nu\sigma
-\int_\Omega\varphi^2\chi_{R}^2\vert\nabla\sigma\vert^2
-2\int_\Omega\varphi^2\chi_{R}\sigma\nabla\chi_{R}\cdot\nabla\sigma.
\end{align*}
Since $\D_\nu\varphi=-\gamma\varphi$, the boundary term reads
\begin{equation}
\int_{\D\Omega}\chi_{R}^2\sigma\varphi^2\D_\nu\sigma=\int_{\D\Omega}\chi_{R}^2v\left[\D_\nu v-\gamma v\right],
\end{equation}
which is nonpositive from~\eqref{LinearEquation}.
From Cauchy-Schwarz inequality, we deduce
\begin{equation}\label{inequality}
\int_\Omega\chi_{R}^2\varphi^2\vert\nabla\sigma\vert^2
\leq 2\sqrt{\int_{\substack{\Omega_{2R}\backslash\Omega_R}}\chi_{R}^2\varphi^2\vert\nabla\sigma\vert^2}\sqrt{\int_\Omega v^2\vert\nabla\chi_{R}\vert^2},
\end{equation}
where $\Omega_R=\Omega\cap\{\vert x\vert \leq R\}$.

The assumption on the growth of the domain at infinity~\eqref{GrowthCondition} implies
\begin{equation}\label{claimBounded}
\int_\Omega v^2\vert\nabla\chi_{R}\vert^2\text{ is bounded, uniformly in }R\geq1.
\end{equation}
From \eqref{inequality}, we deduce that
$\int_\Omega \chi_{R}^2\varphi^2\vert\nabla\sigma\vert^2$ is uniformly bounded. Hence,
\begin{equation}
\lim\limits_{R\to+\infty}\left[\int_{\substack{\Omega_{2R}\backslash\Omega_R}}\chi_{R}^2\varphi^2\vert\nabla\sigma\vert^2\right] =0.
\end{equation}
At the limit as $R\to+\infty$ in~\eqref{inequality}, we find
${
\int_\Omega\varphi^2\vert\nabla\sigma\vert^2\leq0.
}$
Hence $\nabla\sigma=0$, which achieves the proof.
\end{proof}

The cornerstone of the proof is that $\sigma\nabla\cdot(\varphi^2\nabla\sigma)\geq0$ implies $\nabla \sigma=0$, where $\sigma:=\frac{v}{\varphi}$. The litterature refers to this property as a \emph{Liouville property}. Originally introduced in \cite{Berestycki1997b}, it has been extensively used and discussed~\cite{Barlow2000,Gazzola,Moschini2005,Villegas2020}. For more details, we refer to our previour paper~\cite[Section~3.2]{Nordmann2019a}. 
Note that this is the only step where \eqref{GrowthCondition} is needed. In our context, this condition is essentially optimal, as recently proved by Villegas~\cite{Villegas2020}.

\paragraph*{}
We are now ready to complete the proof of~\autoref{th:Unbounded_Domains}.
\begin{proof}[Proof of \autoref{th:Unbounded_Domains}, second and third statements]
Let $\varphi$ be a principal eigenfunction associated with $\lambda_\gamma$. Let us first prove that $v_i:=\D_{x_i}u$ is a constant multiple of $\varphi$, for all $i=1,\dots,n$. First, since $u$ is bounded, classical global Schauder's estimates guarantee that all the $v_i$ are bounded. 
Then, differentiating \eqref{EquationSemilineaire_Unbounded}, we find that $v_i$ satisfies the first equation in~\eqref{LinearEquation}.
Now, we show that all the $v_i$ satisfy the boundary condition in~\eqref{LinearEquation_Bounary}.
Notice that if $v_i$ indeed satisfies~\eqref{LinearEquation_Bounary}, then \autoref{schrodinger2} implies that $v_i$ is a constant multiple of $\varphi$ and therefore $\D_\nu v_i+\gamma v_i=0$. Hence, we know in general that $v_i\left(\D_\nu v_i+\gamma v_i\right)\geq0$. However, \autoref{Lemma_Geom_Proof_Intermediate} implies
\begin{equation}
\sum\limits_{i=1}^n v_i(\D_\nu v_i+\gamma v_i)\leq 0.
\end{equation}
We deduce that $v_i\left(\D_\nu v_i+\gamma v_i\right)=0$. Applying~\autoref{schrodinger2}, we deduce that $v_i$ is a constant multiple of $\varphi$.

The previous step implies that there exists a constant vector $\mathbf{C}\in\R^n$ such that $\nabla u\equiv \mathbf{C}\varphi$. Hence, $u$ is a planar function. If the domain is not a straight cylinder, there exists $x_0\in\D\Omega$ such that $\nu(x_0)\cdot \mathbf{C}\neq 0.$
Then, the Neumann boundary condition in \eqref{EquationSemilineaire_Unbounded} implies that $\nabla u(x_0)=0$. Since $\varphi>0$ on $\overline{\Omega}$, we deduce that $\mathbf{C}=0$, hence $u$ is constant. It completes the proof of the second statement of \autoref{th:Unbounded_Domains}.

Assume that $\Omega$ is a straight cylinder of the form $\R\times\omega$, $\omega\subset\R^{n-1}$, and that $u$ only depends on $x_1$. Since $u'$ is a constant multiple of $\varphi>0$, $u$ is monotonic.
In addition, $u$ is bounded and therefore has a limit$z^+$ when $x_1\to+\infty$. Setting $u_n(x_1)=u(x_1+n)$ and using classical elliptic estimates, we can extract a subsequence that $C^2_{loc}$-converges to a solution $u_\infty$ of \eqref{EquationSemilineaire_Unbounded} which is Robin-curvature-stable. From $u_\infty\equiv z^+$, we deduce that $z^+$ must be a stable root of $f$. Identically, when $x_1\to-\infty$, $u$ converges to a stable root of $f$, denoted $z^-$. 

If $z^+=z^-$, then $u$ is constant. Let us assume $z^-\neq z^+$,  and fix $M>0$. Multiplying $-u''=f(u)$  by $u'$ and integrating on ${x_1\in[-M,M]}$ gives
\begin{equation}
\frac{1}{2}\left(u'(-M)^2-u'(M)^2\right)=\int_{u(-M)}^{u(M)} f.
\end{equation}
Since $u'(\pm\infty)=0$ (indeed, $u'$ is integrable and $u''$ is bounded), when $M$ goes to $+\infty$ we obtain $\int_{z^-}^{z^+}f=0$. The proof of the third statement of \autoref{th:Unbounded_Domains} is thereby complete.
\end{proof}

\subsection{Proof of \autoref{th:CoroInstablePRELIMINARY_Unbounded}}
Let $u$ be a solution of~\eqref{EquationSemilineaire_Unbounded}.
Notice that~\autoref{th:Lemma_lambdaGammaInequalityBreakDown} remains true when the domain is unbounded. It implies that
\begin{equation}
a \lambda_\gamma \geq (a-1)\lambda_0+\mu_{a\gamma}- \sup f',\qquad \forall a\geq 1.
\end{equation}
If $u$ is not constant, then \autoref{th:Unbounded_Domains} implies that $\lambda_\gamma\leq 0$. Then, the proof of \autoref{th:CoroInstablePRELIMINARY} directly follows from the above inequality.

\subsection{Proof of \autoref{th:CoroStable_Unbounded}}

Let $u$ be a solution of~\eqref{EquationSemilineaire_Hetero_Unbounded} and set $v_i:=\D_{x_i} u$, for all $i\in\{1,\dots,n\}$. 
Differentiating \eqref{EquationSemilineaire_Hetero_Unbounded} with respect to $x_i$, we find that $v_i$ satisfies
\begin{equation}\label{linearization_Hetero_2}
    -\Delta v_i - \D_u f(x,u)v_i- \D_{x_i} f(x,u)=0\quad\text{in }\Omega.
\end{equation}
Multiplying by $\chi_R^2v_i$ (we recall that $\chi_R$ is defined in~\eqref{DefCutOff}) and using the divergence theorem, we find
\begin{equation*}
-\int_{\D\Omega}\chi_R^2 v_i\D_\nu v_i+\int_\Omega\nabla\left[\chi_R^2 v_i\right]\cdot \nabla v_i= \int_\Omega \D_u f(x,u)v_i^2\chi_R^2+\int_\Omega \D_{x_i} f(x,u)v_i\chi_R^2
\end{equation*}
Summing the above equation for $i\in\{1,\dots,n\}$, and using that $\left\vert\nabla\left[\chi_R v_i\right]\right\vert^2 = \nabla\left[\chi_R^2 v_i\right]\cdot\nabla v_i+\vert\nabla\chi_R\vert^2 v_i^2$, that $v_i\D_\nu v_i= \frac{1}{2}\D_\nu v_i^2$ , and that 
$\vert \nabla u\vert^2=v_1^2+\dots+v_n^2$, we obtain
\begin{equation}\label{Preuve_Unbounded_CoroStable_1}
\begin{aligned}
&-\frac{1}{2}\int_{\D\Omega}\chi_R^2 \D_\nu\left[\vert\nabla u\vert^2\right] +\sum\limits_{i=1}^n\int_\Omega \vert\chi_R^2 \nabla v_i\vert ^2
\\
&\qquad =
\int_\Omega \vert\nabla\chi_R\vert^2\vert\nabla u\vert^2
\int_\Omega \chi_R^2\D_u f(x,u)\vert\nabla u\vert^2+\int_\Omega \chi_R^2 \nabla_x f(x,u)\cdot \nabla u.
\end{aligned}
\end{equation}

Let us first estimate the left hand-side of the above expression. We use \autoref{LemmeIntermediaireVariations} to deduce that $\frac{1}{2}\D_\nu\left[\vert\nabla u\vert^2\right]\leq \gamma \vert\nabla u\vert^2=\gamma\sum_{i=1}^n v_i^2$. Therefore, recalling the definitions of $\mu_\gamma$ and $\mathcal{G}_\gamma$ from \autoref{Def:CurvatureGamma}, we have
\begin{equation}\label{Preuve_Unbounded_CoroStable_2}
\begin{aligned}
-\frac{1}{2}\int_{\D\Omega}\chi_R^2\D_\nu\left[\vert\nabla u\vert^2\right] 
+\sum\limits_{i=1}^n\int_\Omega \vert \chi_R^2\nabla v_i\vert ^2
\geq \sum_{i=1}^n \mathcal{G}_\gamma(\chi_R^2v_i)
&\geq \mu_\gamma^R \sum_{i=1}^n \int_\Omega \chi_R^2 v_i^2\\
&= \mu_\gamma^R \int_\Omega\chi_R^2\vert\nabla u\vert^2.
\end{aligned}
\end{equation}

Now, we estimate the right-hand side of~\eqref{Preuve_Unbounded_CoroStable_1} using the divergence theorem and that $u$ is a solution of~\eqref{EquationSemilineaire_Hetero_Unbounded}:
\begin{align*}
&\int_\Omega \vert\nabla\chi_R\vert^2\vert\nabla u\vert^2+
\int_\Omega  \chi_R^2\D_u f(x,u)\vert\nabla u\vert^2dx
+\int_\Omega  \chi_R^2\nabla_x f(x,u)\cdot \nabla u\ dx\\
&\qquad=\int_\Omega \vert\nabla\chi_R\vert^2\vert\nabla u\vert^2+\int_\Omega  \chi_R^2\nabla\big[f(x,u(x))\big]\cdot \nabla u\ dx
\\
&\qquad=\int_\Omega \vert\nabla\chi_R\vert^2\vert\nabla u\vert^2+\int_{\D\Omega} \chi_R^2 f(x,u(x))\D_\nu u-\int_\Omega f(x,u(x)) \nabla\cdot \big[\chi_R^2\nabla  u\big]\ dx
\\
&\qquad=\int_\Omega \vert\nabla\chi_R\vert^2\vert\nabla u\vert^2+ \int_\Omega  \chi_R^2 f(x,u(x))^2dx-2\int_\Omega \chi_R f(x,u)\nabla \chi_R\cdot\nabla u,
\shortintertext{and since $-\nabla \chi_R\cdot\nabla u\leq \vert \nabla \chi_R\vert \vert\nabla u\vert$,}
 &\qquad\leq \int_\Omega\left(\vert \nabla \chi_R\vert \vert\nabla u\vert + f(x,u)\chi_R\right)^2 dx.
\end{align*}
We complete the proof of \autoref{th:CoroStable_Unbounded} by plugging the above inequality and~\eqref{Preuve_Unbounded_CoroStable_2} into~\eqref{Preuve_Unbounded_CoroStable_1}.

\subsection{Proof of \autoref{th:Coro_Estimate_Unbounded_Limsup}}

First, notice that, if $u$ is not constant, the strong maximum principle implies that $u$ is not constant on any open set, and therefore $\int_{\Omega_R} \vert \nabla u\vert^2>0$ for all $R>0$. Hence the expression in \autoref{th:Coro_Estimate_Unbounded_Limsup} is well defined.

From~\eqref{CacciopolyInequality_Unbounded}, we have that
\begin{align*}
&\mu_\gamma^R \int_\Omega\chi_R^2\vert\nabla u\vert^2
\leq 2\int_\Omega\vert \nabla \chi_R\vert^2 \vert\nabla u\vert^2+2\int_\Omega\chi_R^2f(x,u)^2dx.
\end{align*}
To prove \autoref{th:Coro_Estimate_Unbounded_Limsup}, it amounts to show that
\begin{equation}
\liminf\limits_{R\to+\infty}\frac{\int_\Omega\vert \nabla \chi_R\vert^2 \vert\nabla u\vert^2}{\int_\Omega\chi_R^2\vert\nabla u\vert^2}=0.
\end{equation}
By contradiction, assume that the above $\liminf$ is strictly positive. We deduce that there exists $\alpha>0$ such that for all $R>0$ we have
\begin{equation}
\frac{\int_{\Omega_{2R}}\vert\nabla u\vert^2}{\int_{\Omega_R}\vert\nabla u\vert^2}\geq \alpha R^2.
\end{equation}
Here, we have used the definition of $\chi_R$ in~\eqref{DefCutOff}, in particular, that $\vert \nabla \chi_R\vert^2 =O\left(R^{-2}\right)$.

This inequality is similar to~\eqref{Contradiction_Preuve_Unbounded} in the proof of \autoref{th:Unbounded_Domains}. However, we have shown there that \eqref{Contradiction_Preuve_Unbounded} leads to a contradiction. The proof is thereby achieved.

\bibliographystyle{abbrv}
\bibliography{/Users/samuelnordmann/Dropbox/Etudes/Bibliographie/library}
\end{document}